\documentclass[12pt]{article}
\usepackage{amssymb,amsmath,amsfonts,amssymb}
\usepackage{amssymb,mathrsfs,bm}
\usepackage[round,comma]{natbib}
\usepackage[latin1]{inputenc}
\usepackage{color}
\usepackage[latin1]{inputenc}
\usepackage{graphicx,float}
\usepackage{subfig}
\usepackage[all]{xy}
\usepackage{soul}
\usepackage{url}
\usepackage{anysize}
\marginsize{1.5cm}{1.5cm}{1.5cm}{3.0cm}

\newtheorem{theorem}{Theorem}

\newtheorem{corollary}[theorem]{Corollary}

\newtheorem{definition}[theorem]{Definition}
\newtheorem{example}[theorem]{Example}

\newtheorem{lemma}[theorem]{Lemma}

\newtheorem{proposition}[theorem]{Proposition}
\newtheorem{remark}[theorem]{Remark}

\newenvironment{proof}[1][Proof]{\noindent\textbf{#1.} }{\ \rule{0.5em}{0.5em}}

\numberwithin{equation}{section} \numberwithin{theorem}{section}

	\title{Comparing the reliability of coherent systems with heterogeneous, dependent and distribution-free components}
\author{ Nuria Torrado \\
	\small{Departamento de An\'{a}lisis Econ\'{o}mico: Econom\'{i}a Cuantitativa}\\
	\small{Universidad Aut\'{o}noma de Madrid, Spain}}
\date{}
	
\begin{document}

		\maketitle

{\bf Abstract: }
In recent years, there has been a growing interest in comparing the reliability of systems with dependent components. 
Most of the recent works consider Archimedean copulas to define the dependency structure of systems and examine specific familiy of distributions.		
In this work, we provide distribution-free results to compare, in the usual stochastic order
under some majorization conditions, 
coherent systems with heterogeneous and dependent components where the dependency structure can be defined by any copula.
Applications of our main results for specific copulas and different families of marginal distributions are also given.

\bigskip


System reliability, coherent systems, dependent components, parallel systems, stochastic comparisons, distorted distributions, copulas,
majorization orders. 

\bigskip

\section{Introduction}
One of the most important issues in reliability theory is to understand 
the reliability of systems and their individual components. 
Mathematical theory of reliability has been well established in \cite{BP96}.
Coherent systems represent a classical structure 
which plays an essential role
in reliability theory.
The general class of coherent systems encompasses different types of systems, such as $k$-out-of-$n$, consecutive $k$-out-of-$n$ and 
consecutive $k$-within-$m$-out-of-$n$, among others.
There are numerous applications for such systems in practice, for instance, telecommunications, oil pipeline systems,
industrial systems (see e.g. \cite{Elsayed}).

In the literature, there exist many researches about stochastic comparisons for the lifetime distributions of specific coherent systems. For instance, 
$k$-out-of-$n$ systems are investigated in \cite{TV12, DXH14, BKM20};
recent works for series and parallel systems are  \cite{LWM20,TO17,NTA-2018}; 
consecutive type systems are discussed in \cite{TK11, TO15, E2020, YCG20, MP20},
to name a few.
Much less studies have
been reported for general coherent systems (see \cite{NASS16, HF19, AS2020, HM2020, TAN21} and the references therein).

In recent years, attention has spread from systems with independent
components to the case of dependent components.
Most of works, on comparing the reliability of systems with dependent components, deal with series and parallel systems where 
the dependency structure used mainly is defined through Archimedean copulas.
In addition,  most of the corresponding results are obtained with respect to specific distributions,
for example, extended exponentials are studied in \cite{Barmalzan-2020}
or to specific families of distributions, namely, scale proportional hazards
(in \cite{FLL18}), 
exponentiated location-scale (in \cite{DK2020}),
among others.
However, to the best of our knowledge,
stochastic
comparisons of systems with dependent components
for a general family of distributions 
have not been considered in the literature so far.

In the present work, we consider coherent systems consisting of $n$ possibly dependent components whose dependency structure can be defined by any copula. 
We denote as $X_i$ the random lifetimes of the $i$th component
with distribution function $F\left( t,\lambda _{i}\right)$ which is
dependent on a parameter $\lambda _{i}$, for $i=1,\ldots ,n$.
Then, we compare the reliability
functions of two coherent systems whose components
are assumed to have dependent and not identically distributed lifetimes
in terms of the parameters vector
${\bm \lambda =(\lambda_1,\ldots,\lambda_n)}$ by using majorization concepts.
{Definition \ref{def_maj} in the Appendix introduces the notion of majorization and related orders used in this manuscript.}
 For
extensive and comprehensive details on the theory of majorization orders and
their applications, one may refer to \cite{Marshall:2011}.

Throughout, we use increasing and decreasing to denote monotone nondecreasing and monotone nonincreasing, respectively. By $a\overset{sign}{=}b$, we mean that $a$ and $b$ have the same sign.

The rest of the paper is organized as follows. We present our main results for 
the class of general coherent systems in Section \ref{sec_main}.
Next, in Section \ref{sec_arq}, we study the case in which the dependence structure of the system is defined by Archimedean copulas for parallel and series systems,
and, in Section \ref{sec_ind}, we investigate the independent case.
Finally, we give some concluding remarks in Section \ref{sec_conclusions}.

\section{General results}\label{sec_main}

In this section, we present our main results for general coherent systems with heterogeneous and dependent components.
{Specifically, let us consider a system consisting of $n$ components possibly dependent with lifetime
$T=\varphi(X_1,\ldots,X_n)$ where $X_i$ represents the lifetime of the $i$th component, for $i=1,\ldots,n$.
The distribution function of this coherent system can be represented as
\begin{equation}\label{eq_def2}
F_D (t)=D(F_1 (t),\ldots, F_n (t)),
\end{equation}
where $F_1,\ldots, F_n$ are the marginal distribution functions of the components
and the distortion function 
$D:[0,1]^{n}\rightarrow[0,1]$ is an increasing function satisfying
$D(0,\ldots,0)=0$ and $D(1,\ldots,1)=1$ (see \cite{NASS16, NA07} and the references therein).
Let us assume that the distribution function of the $i$th component
depends on a parameter $\lambda_i$, i.e., $F_i = F(\lambda_i)$, for $i=1,\ldots,n$. Then, \eqref{eq_def2} can be rewritten as
\[F_D (t)=D(F (t,\lambda_1),\ldots, F (t,\lambda_n)).\]
Similarly, there is an expression for the respective reliability function
\[\bar{F}_{D}(t)=\bar{D}(\bar{F}(t,\lambda_1),\ldots,\bar{F}(t,\lambda_n)),\]
where $\bar{D}(u_1,\ldots,u_n)=1-D(1-u_{1},\ldots,1-u_{n})$
is the multivariate dual distortion function.
When the components are dependent and identically distributed, i.e., $F_{1}=\cdots =F_n =F$, then $D(u,\ldots,u)=d(u)$ is an univariate distortion function such as $d:[0,1]\rightarrow[0,1]$ is increasing satisfying $d(0)=0$ and $d(1)=1$. Therefore, in this particular case, the distribution function of a coherent system is given by $F_{d}(t)=d(F(t))$ (see \cite{NASS13} for more details).}

{Let us define now the  usual stochastic order.
	We say that $X$ is smaller than $Y$ in the {\it usual stochastic order} (shortly written as $X\leq_{ST}Y$) if $F\geq G$ (i.e. this inequality holds for any $t$). This condition is equivalent to  $\bar F\leq \bar G$. For more details on stochastic comparisons see \cite{BRM-2015}.}
The following theorem provides the usual stochastic ordering between two coherent systems,
which are formed by two sets of heterogeneous and possibly dependent random variables
$(X_1,\ldots,X_n)$ and $(Y_1,\ldots,Y_n)$, respectively,
in terms of
the weak submajorization and supermajorization between the parameters vector of marginal distributions.

\begin{theorem}\label{th01}
	Let $F_{D_{1}}=D_{1}\left( F\left( \lambda _{1}\right) ,\ldots
	,F\left( \lambda _{n}\right) \right) $ and $F_{D_{2}}=D_{2}\left( G\left(
	\lambda _{1}^{\ast }\right) ,\ldots ,G\left( \lambda _{n}^{\ast }\right)
	\right) $. Suppose that $D_{1}$ is Schur-concave, $D_{1}\leq D_{2}$, $F\geq
	_{ST}G$ and $F\left(t, \lambda\right) $ is increasing (decreasing) and
	concave in $\lambda$. If 
	${\bm\lambda} \overset{w}{\geq}{\bm\lambda}^{*}
	\left( {\bm\lambda }\geq _{w}{\bm\lambda }^{\ast }\right) $
	then $F_{D_{1}}\geq _{ST}F_{D_{2}}$.
\end{theorem}

\begin{proof} 
	Since $F\left( \lambda _{i}\right) $ is increasing (decreasing) and
	concave in $\lambda _{i}$, from Theorem 5.A.2 (ii or iv) in \cite{Marshall:2011}, we
	have $\left( F\left( \lambda _{1}\right) ,\ldots ,F\left( \lambda
	_{n}\right) \right) \overset{w}{\geq }\left( F\left( \lambda _{1}^{\ast
	}\right) ,\ldots ,F\left( \lambda _{n}^{\ast }\right) \right) $ if
	${\bm\lambda} \overset{w}{\geq}{\bm\lambda}^{*}
	\left( {\bm\lambda }\geq _{w}{\bm\lambda }^{\ast }\right) $. 
	Then, from Theorem 3.A.8. in \cite{Marshall:2011},
	we get that 
	\[ D_{1}\left( F\left( \lambda _{1}\right) ,\ldots
	,F\left( \lambda _{n}\right) \right) \leq D_{1}\left( F\left( \lambda
	_{1}^{\ast }\right) ,\ldots ,F\left( \lambda _{n}^{\ast }\right) \right) \]
	since $D_{1}$ is increasing (by definition) and Schur-concave (by
	assumption). Now, it is clear that $D_{1}\left( F\left( \lambda _{1}^{\ast
	}\right) ,\ldots ,F\left( \lambda _{n}^{\ast }\right) \right) \leq
	D_{2}\left( F\left( \lambda _{1}^{\ast }\right) ,\ldots ,F\left( \lambda
	_{n}^{\ast }\right) \right) $ by assumption $D_{1}\leq D_{2}$. Finally, from
	Proposition 2.3(i) in \cite{NASS16}, we know that 
	\[ D_{2}\left( F\left(
	\lambda _{1}^{\ast }\right) ,\ldots ,F\left( \lambda _{n}^{\ast }\right)
	\right) \leq D_{2}\left( G\left( \lambda _{1}^{\ast }\right) ,\ldots
	,G\left( \lambda _{n}^{\ast }\right) \right) \]
	if $F\left( \lambda
	_{i}^{\ast }\right) \leq G\left( \lambda _{i}^{\ast }\right) $ for $%
	i=1,\ldots ,n$, which holds by the assumption $F\geq _{ST}G$. This complete
	the proof.
\end{proof}

\begin{remark} 
	Observe that, if $F\left( \lambda\right) $ is not increasing
	nor decreasing in $\lambda$, but the rest of the assumptions of Theorem
	\ref{th01} holds ($D_{1}$ is Schur-concave, $D_{1}\leq D_{2}$, $F\geq _{ST}G$
	and $F\left( \lambda\right)$ concave in $\lambda$),
	then ${\bm \lambda} \overset{m}{\geq }{\bm\lambda }^{\ast }$ implies $%
	F_{D_{1}}\geq _{ST}F_{D_{2}}$.
\end{remark}

{
\begin{remark} 
It is worth mentioning that, combining Propositions 2.2(i) and 2.3(i) in \cite{NASS16}, we know that if $D_{1}\leq D_{2}$ $F\geq_{ST}G$ and $F\left(t, \lambda\right) $ is increasing (decreasing) in $\lambda$, then 
$\lambda_i\leq\lambda_{i}^{*}$ ($\lambda_i\geq\lambda_{i}^{*}$) for all $i=1,\ldots,n$ implies $F_{D_{1}}\geq _{ST}F_{D_{2}}$.
However, it is well known that the condition $\lambda_i\leq\lambda_{i}^{*}$ for all $i=1,\ldots,n$ is a stronger condition than the majorization order, which, in turn, implies both  
weak submajorization and supermajorization orders.
\end{remark}}

Theorem \ref{th01} can be applied to any coherent system with possibly dependent
components. In particular, for series systems defined as $X_{1:n}=\min
\left( X_{1},\ldots ,X_{n}\right) $, the reliability function is 
\begin{equation*}
	\bar{F}_{1:n}(t)=\widehat{C}\left( \bar{F}_{1}(t),\ldots ,%
	\bar{F}_{n}(t)\right) ,
\end{equation*}%
where $\widehat{C}$ is the survival copula, then in this case, $\bar{D}=%
\widehat{C}$. Hence, 
\[F_{1:n}(t)=1-\bar{F}_{1:n}(t)=1-\widehat{C}\left(
1-F_{1}(t),\ldots ,1-F_{n}(t)\right) ,\] 
that is, its multivariate distortion
function is 
$D\left( u_{1},\ldots ,u_{n}\right) =1-\widehat{C}\left(
1-u_{1},\ldots ,1-u_{n}\right) $. 
Analogously, for parallel systems defined
as $X_{n:n}=\max \left( X_{1},\ldots ,X_{n}\right) $, the distribution
function is 
\begin{equation}\label{eq_max}
	F_{n:n}(t)=C\left( F_{1}(t),\ldots ,F_{n}(t)\right) ,
\end{equation}%
where $C$ is the distributional copula, therefore, $D=C$. In the following
example, we show how Theorem \ref{th01} can be applied to compare the reliability
between series and parallel systems.

\begin{example}\label{ex_MPHR} 
	Let two random vectors $\left( X_{1},X_{2}\right) $ and 
	{$\left(Y_{1},Y_{2}\right)$} such as $X_{i}\sim F\left( t,\lambda _{i}\right) $ and 
	{$Y_{i}\sim G\left( t,\lambda _{i}^{\ast }\right) $}
	for $i=1,2$. 
	Assume the
	parallel and series systems with lifetimes 
	$X_{2:2}=\max \left(
	X_{1},X_{2}\right) $ and {$Y_{1:2}=\min \left( Y_{1},Y_{2}\right)$}, respectively, with two
	dependent components having the following Farlie-Gumbel-Morgenstern (FGM) copula:%
	\begin{equation*}
		C_{FGM}\left( u_{1},u_{2}\right) =u_{1}u_{2}\left( 1+\theta \left(
		1-u_{1}\right) \left( 1-u_{2}\right) \right),
	\end{equation*}%
	for $(u_{1},u_{2})\in[0,1]^{2}$ and $\theta \in \left[ -1,1\right] $. Then, the distribution functions are 
	$F_{X_{2:2}}=D_{1}\left( F\left( \lambda _{1}\right) ,F\left( \lambda
	_{2}\right) \right) $ and 
	$G_{Y_{1:2}}=D_{2}\left( G\left( \lambda _{1}^{\ast
	}\right) ,G\left( \lambda _{2}^{\ast }\right) \right) $
	where $D_{1}\left(
	u_{1},u_{2}\right) =C_{FGM}\left( u_{1},u_{2}\right) $ 
	and 
	\begin{eqnarray*}
		D_{2}\left( u_{1},u_{2}\right) &=&1-C_{FGM}\left( 1-u_{1},1-u_{2}\right) \\
		&=&1-\left( 1-u_{1}\right) \left( 1-u_{2}\right) \left( 1+\theta
		u_{1}u_{2}\right) .
	\end{eqnarray*}%
	Then, we have%
	\begin{eqnarray*}
		D_{2}-D_{1} &=&1-\left( 1-u_{1}\right) \left( 1-u_{2}\right) \left( 1+\theta
		u_{1}u_{2}\right) -\left( u_{1}u_{2}\left( 1+\theta \left( 1-u_{1}\right)
		\left( 1-u_{2}\right) \right) \right) \\
		&=&u_{1}\left( 1-u_{2}\right) +u_{2}\left( 1-u_{1}\right) +2\theta
		u_{1}u_{2}\left( u_{2}+u_{1}-u_{1}u_{2}-1\right) \\
		&=&u_{1}\left( 1-u_{2}\right) +u_{2}\left( 1-u_{1}\right) +2\theta
		u_{1}u_{2}\left( 1-u_{2}\right) \left( u_{1}-1\right) \\
		&=&u_{1}\left( 1-u_{2}\right) +u_{2}\left( 1-u_{1}\right) \left( 1-2\theta
		u_{1}\left( 1-u_{2}\right) \right) \geq 0,
	\end{eqnarray*}%
	when $-1\leq \theta \leq 0$. It is known that the FGM copula is
	Schur-concave for any $\theta\in[-1,1]$ (see \cite{N06}), i.e., $D_{1}=C_{FGM}$ is Schur-concave. Then,
	from Theorem \ref{th01}, we have that $X_{2:2}\geq _{ST}Y_{1:2}$ 
	if ${\bm\lambda} \overset{w}{\geq }%
	{\bm\lambda }^{\ast }\left( {\bm \lambda }\geq _{w}{\bm\lambda }^{\ast
	}\right) $, when $-1\leq \theta \leq 0$,
	for any distribution function $F\left( t,\lambda\right) $
	increasing (decreasing) and concave in $\lambda$. 
	In particular, let us assume that $F\left( t,\lambda\right) $ {and $G\left( t,\lambda^{*}\right) $}
	belong to the MPHR family which
	is defined
	in \eqref{MPHR} where $\bar H_{F}(t)=e^{- t}$ {and $\bar H_{G}(t)=e^{- 2t}$, respectively. Observe that $F\geq_{ST}G$}. 
	From Lemma \ref{lema01}(iii), we know that 
	$F\left( t,\lambda\right) $ is increasing and concave in $\lambda$
	for any $\alpha>0$.
		We take $\theta=-0.5$, ${\bm \lambda}=(0.7, 1.5)$ and ${\bm \lambda}^{*}=(1.1, 1.3)$. Note that ${\bm\lambda} \overset{w}{\geq }{\bm\lambda }^{\ast }$.
		{Therefore, all the conditions on Theorem \ref{th01} are satisfied}.
	Figure \ref{Figr_MPHR} plots the distribution functions of $X_{2:2}$ (solid line) and $Y_{1:2}$ (dashed line)
	for $\alpha=0.8$ (left) and $\alpha=1.6$ (right). 
	It is evident that $X_{2:2}\geq_{ST}Y_{1:2}$, which is according to
	Theorem \ref{th01}.
\end{example}

\begin{figure}[H]
	\centering
	\includegraphics[width=6 cm]{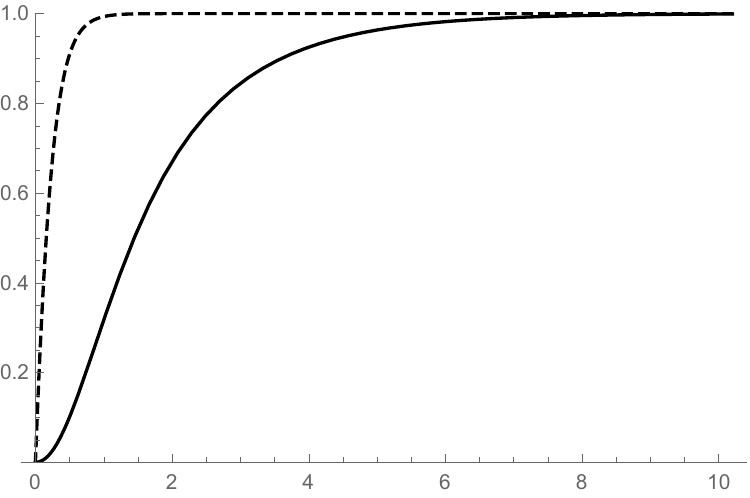}
	\includegraphics[width=6 cm]{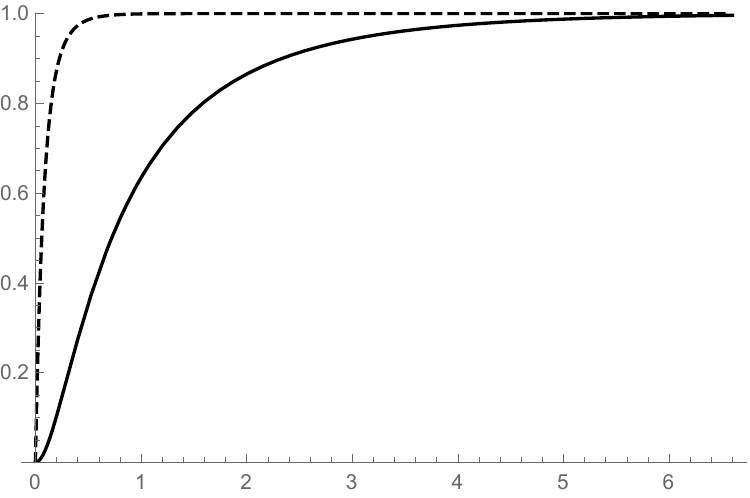}
	\caption{Plots of the distribution functions of $X_{2:2}$ (solid line)
		and $Y_{1:2}$  (dashed line)
	considered in Example \ref{ex_MPHR}
	with $\alpha=0.8$ (left) and $\alpha=1.6$ (right).}
	\label{Figr_MPHR}
\end{figure}

We can also obtain results for parallel systems consisting of $n$ dependent
components assembled by a copula different than the family of FGM copulas. In
particular, we can consider two parallel systems
whose dependence structure is defined by a distributional
Cuadras-Aug\'{e} bivariate copula given by 
\[{C}_{\theta}(u,v)=[min(u,v)]^{\theta}\, [uv]^{1-\theta},	\]
for $\theta\in[0,1]$ and $(u,v)\in[0,1]^{2}$. Observe that
the parameter $\theta $ measures the degree of stochastic dependence between
the two components. It is known that if $0\leq \theta _{1}\leq \theta
_{2}\leq 1$ (see \cite{N06}), then $C_{\theta _{1}}(u,v)\leq C_{\theta
	_{2}}(u,v)$. Then, if $F\geq _{ST}G$ and $F\left( \lambda\right) $ is
increasing (decreasing) and concave in $\lambda$, from Theorem \ref{th01}, we
have that ${\bm\lambda} \overset{w}{\geq }{\bm\lambda }^{\ast }\left( {%
	\bm\lambda }\geq _{w}{\bm\lambda }^{\ast }\right) $ implies $X_{n:n}\geq
_{ST}Y_{n:n}$. This example shows that the result given in Theorem \ref{th01} can
also used to compare systems with different levels of dependency. Similar
results can be obtained, for example, for the  
Plackett family of copulas which
is positively ordered or the Gumbel-Barnett family of copulas which is
negatively ordered, among others (see \cite{N06}).
The family of Archimedean copulas for parallel and series systems is studied in Section \ref{sec_arq} in detail.

In addition, Theorem \ref{th01} can be used to compare other types of coherent
systems. Let us see another example on how Theorem \ref{th01} can be applied.

\begin{example} 
	Let us consider two
	systems $T_{1}=\max \left( X_{1},X_{2},X_{3}\right) $ and $T_{2}=\min \left(
	X_{1},\max \left( X_{2},X_{3}\right) \right) $, both with independent
	components. From Table 1 in  \cite{NA07}, we know that $D_{1}\leq
	D_{2}$ where%
	\begin{equation*}
		D_{1}\left( u_{1},u_{2},u_{3}\right) =1-\bar{D}_{1}\left(
		1-u_{1},1-u_{2},1-u_{3}\right) =u_{1}u_{2}u_{3}
	\end{equation*}%
	and%
	\begin{equation*}
		D_{2}\left( u_{1},u_{2},u_{3}\right) =1-\bar{D}_{2}\left(
		1-u_{1},1-u_{2},1-u_{3}\right) =u_{1}+u_{2}u_{3}-u_{1}u_{2}u_{3}.
	\end{equation*}%
	It is easy to check $D_{1}$ is Schur-concave. Then, from Theorem \ref{th01}, 
	we get $T_{1}\geq _{ST}T_{2}$ if ${\bm\lambda} \overset{w}{\geq }{\bm\lambda }^{\ast
	}\left( {\bm\lambda }\geq _{w}{\bm\lambda }^{\ast }\right) $ for any
	distribution function $F\left( t,\lambda\right) $ increasing
	(decreasing) and concave in $\lambda$ as those studied in Lemma \ref{lema01}.
\end{example}

If all the conditions in Theorem \ref{th01} are satisfied except concavity of
$F$ with respect to a parameter, then 
coherent systems may not be ordered,
as we show in the following example.

\begin{example}\label{ex_power}
		Let two random vectors $\left( X_{1},X_{2}\right) $ and 
	$\left(X_{1}^{*},X_{2}^{*}\right) $ such as $X_{i}\sim F\left( t,\lambda _{i}\right) $ and 
	$X_{i}^{*}\sim F\left( t,\lambda _{i}^{\ast }\right) $,
	where $F$ is the power generalized Weibull distribution such as
	$$F\left( t,\lambda _{i}\right) =1-e^{1-(1+t^2)^{1/\lambda _{i}}}
	\quad \textup{with } \lambda _{i}>0 \textup{ and } t>0,$$
	for $i=1,2$. 
	After some computations, we get 
	\[\frac{\partial^2 F\left( t,\lambda\right)}{\partial \lambda^2 }
	\overset{sign}{=}
	2\lambda +\log (1+t^2)(1-(1+t^2)^{1/\lambda}).\]
	The above function changes the sign, so $F$ is not convex nor concave
	with respect to $\lambda$.
	We consider two parallel systems $X_{2:2}$ and $X^{*}_{2:2}$
	under two cases, when the components are independent and when they are assembled by a Clayton copula defined as follows
	\[C(u_1,u_2)=(u_1^{-\theta}+u_2^{-\theta}-1)^{-1/\theta}, \quad \theta>0
	\textup{ and } u_1, u_2\in[0,1]. \]
	For simplicity, we assume $D_1=D_2=C$. 	
	On the other hand, we consider that $\bm{\lambda}=(0.2, 0.6)$ and
	$\bm{\lambda}^{*}=(0.35, 0.45)$. It is easy to verify that 
	${\bm\lambda} \overset{m}{\geq }{\bm\lambda }^{\ast }$.	
	Figure \ref{Figr_power} plots the distribution functions of the two parallel systems for the independent case (left) and the dependent case (right) for $\theta=0.08$. In both cases, the distribution functions intersect and therefore the parallel systems are not ordered
	since $F$ is not concave in $\lambda$.	
\end{example}

\begin{figure}[H]
	\centering
	\includegraphics[width=6 cm]{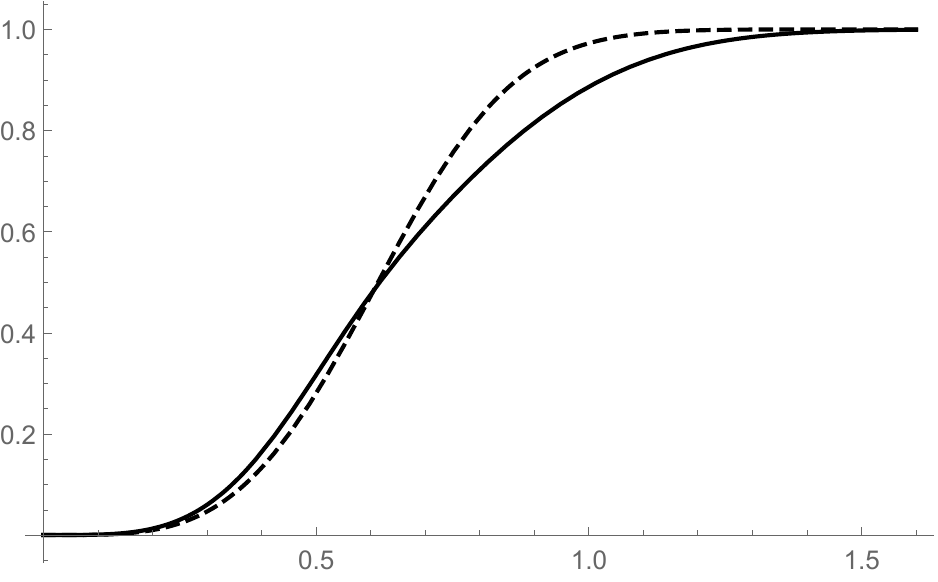}
	\includegraphics[width=6 cm]{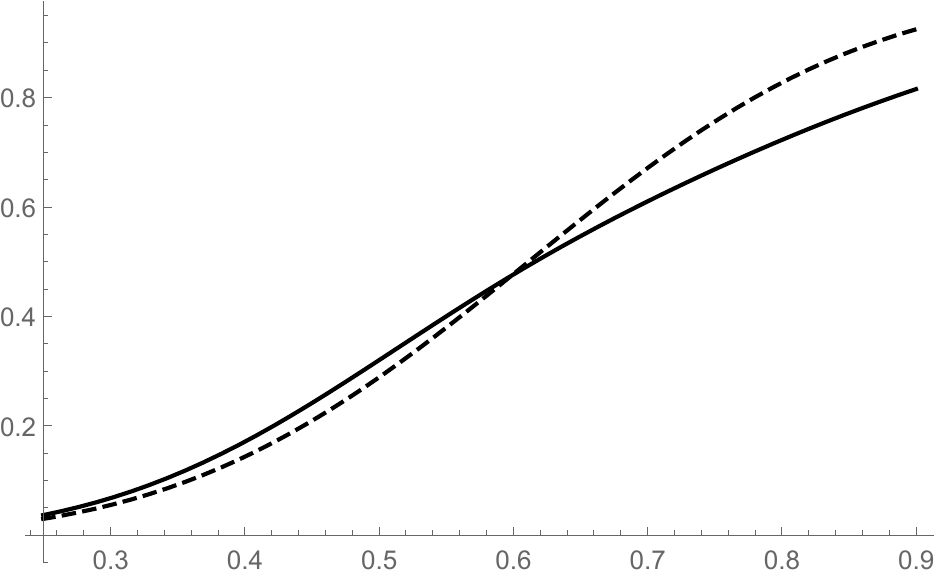}
	\caption{Plots of the distribution functions of $X_{2:2}$ (solid line)
		and $X^{*}_{1:2}$  (dashed line)
		considered in Example \ref{ex_power}
		for independent  (left) and dependent  (right) components.}
	\label{Figr_power}
\end{figure}

It is worth noting that Theorem \ref{th01} can be easily applied to parallel systems
due to \eqref{eq_max} and the fact that most copulas (at least those commonly used) are
Schur-concave. 
{Specifically, FGM, Plackett family, Gumbel-Barnett family and Cuadras-Aug\'{e} bivariate copula are
	Schur-concave (see \cite{DS-2003})
	and they are also positive or negative ordered (see \cite{N06}).
	Thus,  the following remark is a direct consequence of Theorem \ref{th01} and Lemma \ref{lema01}}.

{\begin{remark}\label{rem_new}
			Let $X_{1},\ldots ,X_{n}$ and $Y_{1},\ldots ,Y_{n}$ be two sets
		of nonnegative random variables with associated distributional copulas
		$C_X$ and $C_Y$, respectively.
		Suppose that $C_X$ is Schur-concave and $C_X\leq C_Y$.
		\begin{itemize}
			\item[i)] If $X_{i}\sim MPHR\left(\lambda
			_{i}\right) $ and $Y_{i}\sim MPHR\left(\lambda _{i}^{\ast
			}\right) $ for $i=1,\ldots ,n$,
			then 	${\bm\lambda} \overset{w}{\geq }{\bm\lambda }^{\ast }$ implies $X_{n:n}\geq _{ST}Y_{n:n}$, for $\alpha>0$.	
			\item[ii)] If $X_{i}\sim SPH\left(\lambda
			_{i},\alpha _{i}\right) $ and $Y_{i}\sim SPH\left(\lambda _{i}^{\ast
			},\alpha _{i}^{\ast }\right) $ for $i=1,\ldots ,n$,
			and $\tilde{r}$  decreasing in $t>0$,
			then $%
			{\bm\lambda} \overset{w}{\geq }{\bm\lambda }^{\ast }$ and ${\bm\alpha }%
			\overset{w}{\geq }{\bm\alpha }^{\ast }$
			imply $X_{n:n}\geq _{ST}Y_{n:n}$.	
			\item[iii)] If $X_{i}\sim$ EK-G$\left(\alpha, \beta
			_{i},\gamma\right) $ and $Y_{i}\sim$ EK-G$\left(\alpha, \beta
			_{i}^{*},\gamma\right) $ for $i=1,\ldots ,n$,
			then	${\bm\beta} \overset{w}{\geq }{\bm\beta }^{\ast }$ implies $X_{n:n}\geq _{ST}Y_{n:n}$,
			for $\alpha>0$ and $\gamma \leq 1$.	
		\end{itemize}
	\end{remark}
}

{The MPHR model, for $\alpha=1$, has been recently applied in Aeronautics
	(see \cite{Diamoutene}), whereas
	the EK-G distribution has been successfully applied to reliability analysis by \cite{Silva-2019}.}

However, for series systems, it can be the case that $D\left(
u_{1},\ldots ,u_{n}\right) =1-\widehat{C}\left( 1-u_{1},\ldots ,1-u_{n}\right) 
$ is not Schur-concave. For instance, for a series system with a FGM
survival copula, its multivariate distortion function is 
\begin{eqnarray}\label{eq_fgm}\nonumber
	D\left( u_{1},u_{2}\right) &=&1-\widehat{C}_{FGM}\left( 1-u_{1},1-u_{2}\right) \\
	&=&1-\left( 1-u_{1}\right) \left( 1-u_{2}\right) \left( 1+\theta
	u_{1}u_{2}\right)
\end{eqnarray}%
and it is easy to see that $D$ is Schur-convex. Therefore we can not use
Theorem \ref{th01} in order to compare two series systems. For this reason we present
the following result.

\begin{theorem}\label{th02}
	Let $F_{D_{1}}=D_{1}\left( F\left( t,\lambda _{1}\right) ,\ldots
	,F\left( t,\lambda _{n}\right) \right) $ and $F_{D_{2}}=D_{2}\left( G\left(
	t,\lambda _{1}^{\ast }\right) ,\ldots ,G\left(t, \lambda _{n}^{\ast }\right)
	\right) $. Suppose that $D_{1}$ is Schur-convex, $D_{1}\geq D_{2}$, $F\leq
	_{ST}G$ and $F\left(t, \lambda\right) $ is decreasing (increasing) and
	convex in $\lambda$. If ${\bm\lambda} \overset{w}{\geq }{\bm\lambda }%
	^{\ast }\left( {\bm\lambda }\geq _{w}{\bm\lambda }^{\ast }\right) $
	then $F_{D_{1}}\leq _{ST}F_{D_{2}}$.
\end{theorem}

\begin{proof} 
	Since $F\left( t,\lambda\right) $ is decreasing (increasing) and
	convex in $\lambda$, from Theorem 5.A.2 (i or iii) in \cite{Marshall:2011}, we
	have $\left( F\left( \lambda _{1}\right) ,\ldots ,F\left( \lambda
	_{n}\right) \right) \geq _{w}\left( F\left( \lambda _{1}^{\ast }\right)
	,\ldots ,F\left( \lambda _{n}^{\ast }\right) \right) $ if 
	${\bm\lambda} 
	\overset{w}{\geq }{\bm\lambda }^{\ast }\left( {\bm\lambda }\geq _{w}{%
		\bm\lambda }^{\ast }\right) $. Then, from Theorem 3.A.8. in \cite{Marshall:2011}, we
	get that 
	\[D_{1}\left( F\left( t,\lambda _{1}\right) ,\ldots ,F\left(t, \lambda
	_{n}\right) \right) \geq D_{1}\left( F\left( t,\lambda _{1}^{\ast }\right)
	,\ldots ,F\left( t,\lambda _{n}^{\ast }\right) \right) \]
	since $D_{1}$ is
	increasing (by definition) and Schur-convex (by assumption). Now, it is
	clear that $D_{1}\left( F\left(t, \lambda _{1}^{\ast }\right) ,\ldots ,F\left(
	t,\lambda _{n}^{\ast }\right) \right) \geq D_{2}\left(t, F\left( \lambda
	_{1}^{\ast }\right) ,\ldots ,F\left(t, \lambda _{n}^{\ast }\right) \right) $
	by assumption $D_{1}\geq D_{2}$. Finally, from Proposition 2.3(i) in
	\cite{NASS16}, we know that 
	\[D_{2}\left( F\left( t,\lambda _{1}^{\ast
	}\right) ,\ldots ,F\left( t,\lambda _{n}^{\ast }\right) \right) \geq
	D_{2}\left( G\left( t,\lambda _{1}^{\ast }\right) ,\ldots ,G\left(t, \lambda
	_{n}^{\ast }\right) \right) \]
	if $F\left( t,\lambda _{i}^{\ast }\right) \geq
	G\left(t, \lambda _{i}^{\ast }\right) $ for $t>0$ and $i=1,\ldots ,n$, which holds by
	the assumption $F\leq _{ST}G$. This complete the proof.
\end{proof} 

\bigskip

Next, we present an example for comparing two series systems consisting of
dependent components which have MPRHR lifetime distributions.

\begin{example}\label{ex_MPRHR}
	Let $X_{1},X_2$ and {$Y_{1},Y_2$} be two sets
	of dependent random variables with $X_{i}\sim MPRHR\allowbreak(\lambda_i,{F})$ and {$Y_{i}\sim MPRHR(\lambda_i,G)$} for $%
	i=1,2$. We assume that the distribution functions of $X_i$ {and $Y_i$ are}
	as defined in \eqref{MPRHR} where $H$ is the distribution function of an exponential random variable with hazard rate $2$ {and $1$} for $i=1,2$, respectively.
	{Note that $F\leq_{ST}G$}.
	From Lemma \ref{lema01}(ii), we know that $F$ is convex with respect to $\lambda$ when $0<\alpha\leq 1$.
	Set $\alpha =0.9$, ${\bm \lambda}=(0.7, 1.5)$ and ${\bm \lambda}^{*}=(1.1, 1.3)$. It is easy to check that ${\bm\lambda} \overset{w}{\geq }{\bm\lambda }^{\ast }$.
	Now, we consider two series systems $X_{1:2}$ and
	{$Y_{1:2}$} with dependent components assembled by the same FGM survival copula with $\theta=-0.5$ {or $\theta=0.5$}, that is, 
	$D_1=D_2=D$ which is defined as  in \eqref{eq_fgm}. 
Then, all the conditions in	Theorem \ref{th02} are satisfied.
	Figure \ref{Figr_MPRHR} plots the distribution functions of $X_{1:2}$ (solid line) and $Y_{1:2}$ (dashed line)
	{for $\theta=-0.5$ (left) and $\theta=0.5$ (right)}. 
	It is evident that $X_{1:2}\leq_{ST}Y_{1:2}$, 
	validating the result in
	Theorem \ref{th02}.
\end{example}

\begin{figure}[H]
	\centering
	\includegraphics[width=6 cm]{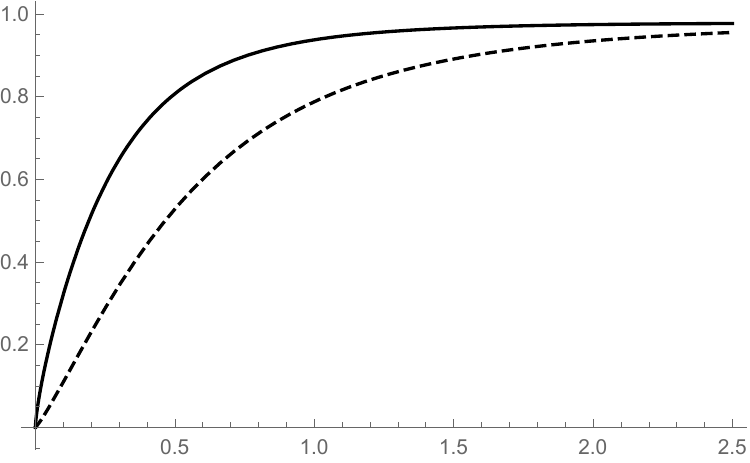}
	\includegraphics[width=6 cm]{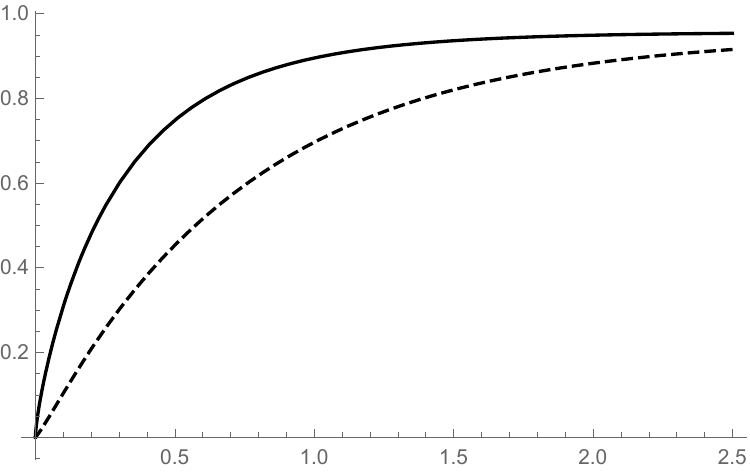}
	\caption{Plots of the distribution functions of $X_{1:2}$ (solid line)
		and $Y_{1:2}$  (dashed line)
		considered in Example \ref{ex_MPRHR}
		for dependent components assembled by FGM survival copulas
		when $\theta=-0.5$  (left) and $\theta=0.5$   (right).}
	\label{Figr_MPRHR}
\end{figure}

Theorem \ref{th02} can be also used to compare systems with different structure. Let us
see an example on how Theorem \ref{th02} can be applied in this case.

\begin{example}\label{ex_04} 
	Let us consider two
	systems
	\[T_{1}=\min \left( X_{1},\max \left( X_{2},X_{3}\right) \right) \]
	and 
	\[T_{2}=\max \left( \min \left( X_{1},X_{2}\right) ,\min \left(
	X_{1},X_{3}\right) ,\min \left( X_{2},X_{3}\right) \right) \] 
	where the
	components of both systems are independent. 
	Observe that $T_2$ is a $2$-out-of-$3$ system.
	From Table 1 in
	\cite{NA07}, we know that $D_{1}\geq D_{2}$ where%
	\begin{equation*}
	D_{1}\left( u_{1},u_{2},u_{3}\right) =1-\bar{D}_{1}\left(
	1-u_{1},1-u_{2},1-u_{3}\right) =u_{1}+u_{2}u_{3}-u_{1}u_{2}u_{3}
	\end{equation*}%
	and%
	\begin{equation*}
	D_{2}\left( u_{1},u_{2},u_{3}\right) =1-\bar{D}_{2}\left(
	1-u_{1},1-u_{2},1-u_{3}\right)
	=u_{1}u_{2}+u_{1}u_{3}+u_{2}u_{3}-2u_{1}u_{2}u_{3},
	\end{equation*}%
	for $(u_{1},u_{2},u_{3})\in [0,1]^3$.
	In order to prove that $D_{1}$ is Schur-convex, we denote the partial
	derivative of $D_{1}$ wih respect to its $k$th argument by $\partial
	_{(k)}D_{1}$. 
	It is easy to check that%
	\begin{equation*}
	\partial _{(1)}D_{1}=1-u_{2}u_{3}\geq \partial _{(2)}D_{1}=u_{3}\left(
	1-u_{1}\right) \geq \partial _{(3)}D_{1}=u_{2}\left( 1-u_{1}\right),
	\end{equation*}%
	if and only if, $u_{2}\leq u_{1}$ and $u_{2}\leq u_{3}$. 
	In other words, if $%
	F\left( t,\lambda\right) $ is decreasing (increasing) in $\lambda$
	and $\lambda _{2}\geq \max \left( \lambda _{1},\lambda _{3}\right)$ $\left(
	\lambda _{2}\leq \min \left( \lambda _{1},\lambda _{3}\right) \right) $,
	then $D_{1}$ is Schur-convex. 
	Therefore, from Theorem \ref{th02}, $T_{1}\leq _{ST}T_{2}$
	whenever ${\bm\lambda} \overset{w}{\geq }{\bm\lambda }^{\ast }\left( {%
		\bm\lambda }\geq _{w}{\bm\lambda }^{\ast }\right) $ for any distribution
	function $F\left( t,\lambda\right) $ decreasing (increasing) and convex
	in $\lambda$.
	Specifically, we assume that $F$ belongs to the TIIHL-RS model defined in
	\eqref{eq_tiil} where $H$ is the distribution function of an exponential random variable with hazard rate $\lambda>0$.
	From Lemma \ref{lema01}(i), we know that $F$ is decreasing and convex in $\alpha$ for any $\lambda>0$.
	We take $\lambda =0.8$, ${\bm \alpha}=(0.5, 1.6, 2.1)$ and ${\bm \alpha}^{*}=(0.9, 1.3, 2.3)$. Observe that ${\bm\alpha} \overset{w}{\geq }{\bm\alpha }^{\ast }$.
	We plot the distribution functions of $T_1$ (solid line) and $T_2$ (dashed line) in Figure \ref{Fig_TIIHL}. 
	It is evident that $T_1\leq_{ST}T_2$, which is according to
	Theorem \ref{th02}.
\end{example} 

\begin{figure}[H]
	\centering
	\includegraphics[width=6 cm]{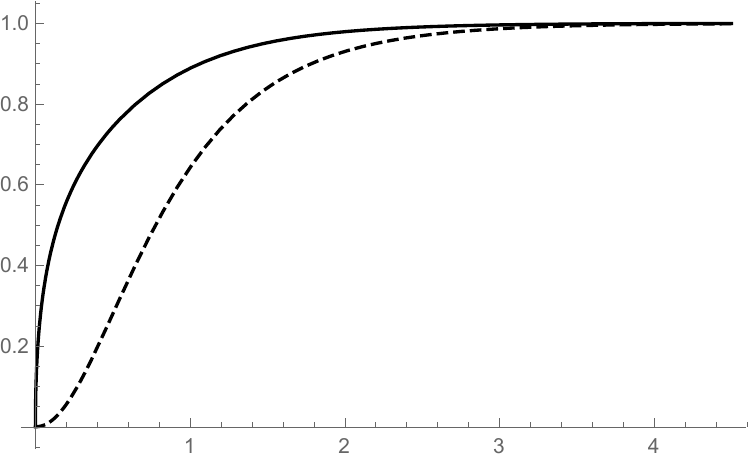}
	\caption{Plots of the distribution functions of  $T_1$ (solid line) and 
		$T_2$ (dashed line) considered in Example \ref{ex_04}.}
	\label{Fig_TIIHL}
\end{figure}

The following example shows that the condition  $F$ 
convex given in the
above Theorem \ref{th02} can not be dropped.
In addition, it shows that log-concavity of the function $F$ with respect to $\lambda$ is not a sufficient condition to obtain the ordering between series systems with dependent components.

\begin{example}\label{ex_02}
	We consider two series systems $X_{1:2}$ and $X^{*}_{1:2}$
	with dependent components assembled by a FGM survival copula. 
	For simplicity, we assume $D_{1}=D_{2}=D$ and $F\overset{ST}{=}G$.
	Then, the associated multivariate distortion function is the same than in 
	\eqref{eq_fgm} and we know that it is Schur-convex.
	Let us assume that the components have inverse Weibull distribution with
	$F(t,\lambda,\alpha)=e^{-(\lambda t)^{-\alpha}}$ for $t>0$ when $\alpha,\lambda>0$. Hence, the distribution functions of the systems are
	\[F_{X_{1:2}}(t)=1-(1-F(t,\lambda_1))(1-F(t,\lambda_2))
	(1+\theta F(t,\lambda_1)F(t,\lambda_2))\]
	and 
	\[F_{X_{1:2}^{*}}(t)=1-(1-F(t,\lambda_1^{*}))(1-F(t,\lambda_2^{*}))
	(1+\theta F(t,\lambda_1^{*})F(t,\lambda_2^{*})).\]
	On the other hand, it is easy to see that $F(t,\lambda,\alpha)$ is increasing in $\lambda$. Taking the second partial derivative of $F(t,\lambda,\alpha)$ with respect to $\lambda$, we get
	\[\frac{\partial^{2} F(t,\lambda,\alpha)}{\partial\lambda^{2}}
	\overset{sign}{=} \alpha-(1+\alpha) (\lambda t)^{\alpha}.\]
	Note that this function changes the sign, therefore $F(t,\lambda,\alpha)$ is not convex nor concave in $\lambda$. 
	Set $\theta=0.5$, $\alpha=1.3$, ${\bm\lambda} =(2,4)$ and ${\bm\lambda}^{*} =(3,3)$. It is easy to check that 
	${\bm\lambda}\overset{m}{\geq}{\bm\lambda}^{*}$.
	Then, we obtain 
	$F_{X_{1:2}}(0.1)=0.0375\geq
	0.0166=F_{X_{1:2}^{*}}(0.1)$ and 
	$F_{X_{1:2}}(0.8)=0.8978\leq
	0.905=F_{X_{1:2}^{*}}(0.8)$, i.e, the distribution functions of the two systems
	intersect.
	In addition, observe that
	$\log F(t,\lambda,\alpha)=-(\lambda t)^{-\alpha}$ is concave in $\lambda$, i.e, 
	$F(t,\lambda,\alpha)$ is log-concave in $\lambda$. Therefore, log-concavity of the function $F(t,\lambda,\alpha)$ in $\lambda$ is not a sufficient condition to obtain the ordering between series systems with dependent components.
\end{example}

{Of course, based on Theorem \ref{th02} and Lemma \ref{lema01}, it is possible to obtain, for series systems, similar comparisons to those in Remark \ref{rem_new}. It is omitted for the sake of brevity.}

\section{Archimedean copulas for extreme OS}\label{sec_arq}

In this Section, we investigate comparisons between parallel
and series systems with dependent components assembled with Archimedean copulas and with lifetime distribution functions dependent on a parameter $\lambda$.
%
{By definition, for a decreasing and continuous function $\psi :[0,\infty ]\rightarrow
\lbrack 0,1]$\ such that $\psi \left( 0\right) =1$\ and $\psi \left( +\infty
\right) =0$\ and $\phi =\psi ^{-1}$\ being the pseudo-inverse,%
\begin{equation}\label{eq_arq}
C_{\psi}\left( u_{1},\ldots ,u_{n}\right) =\psi \big( \phi \left( u_{1}\right)
+\cdots +\phi \left( u_{n}\right) \big), 
\end{equation}%
for $u_{i}\in \lbrack 0,1]$ with $i=1,\ldots ,n$, is said to be an
Archimedean copula with generator $\psi $ if $\left( -1\right) ^{k}\psi
^{\lbrack k]}\left( x\right) \geq 0$\ for $k=0,\ldots ,n-2$\ and $\left(
-1\right) ^{n-2}\psi ^{\lbrack n-2]}$ is decreasing and convex. The
Archimedean copula family includes many well-known copulas such as
independence (product) copula, Clayton copula, and Ali-Mikhail-Haq
copula, among others. 
For more detailed study on the
properties of copulas, one may refer to \cite{N06}.}

\subsection{Parallel systems with Archimedean copulas}

A direct consequence of Theorem \ref{th01} is the following
result.

\begin{corollary}\label{cor01}
	Let $X_{1},\ldots ,X_{n}$ and $Y_{1},\ldots ,Y_{n}$ be two sets
	of nonnegative random variables with $X_{i}\sim F\left( t,\lambda
	_{i}\right) $ and $Y_{i}\sim G\left( t,\lambda _{i}^{\ast }\right) $ for $%
	i=1,\ldots ,n$, with associated Archimedean copulas $C_{\psi _{1}}$ and $%
	C_{\psi _{2}}$, respectively. Suppose that 
	$C_{\psi _{1}}\leq C_{\psi _{2}}$,
    $F\geq _{ST}G$ and $F\left( t,\lambda\right) $ is
	increasing (decreasing) and concave in $\lambda $. If ${\bm\lambda} 
	\overset{w}{\geq }{\bm\lambda }^{\ast }\left( {\bm\lambda \geq }_{w}{%
		\bm\lambda }^{\ast }\right) $ then $X_{n:n}\geq _{ST}Y_{n:n}$.
\end{corollary}

\begin{proof}
	From \eqref{eq_max}, we know that the distribution functions of $X_{n:n}$ and $%
	Y_{n:n}$ are 
	\[F_{X_{n:n}}(t)=D_{1}\left( F\left(t, \lambda _{1}\right) ,\ldots
	,F\left(t, \lambda _{n}\right) \right) =C_{\psi _{1}}\left( F\left(t, \lambda
	_{1}\right) ,\ldots ,F\left( t,\lambda _{n}\right) \right) \] 
	and 
	\[
	F_{Y_{n:n}}(t)=D_{2}\left( G\left(t,\lambda _{1}^{\ast }\right) ,\ldots ,G\left(t,
	\lambda _{n}^{\ast }\right) \right) =C_{\psi _{2}}\left( G\left( t,\lambda _{1}^{\ast }\right)
	,\ldots ,G\left(t, \lambda _{n}^{\ast }\right) \right),\] 
	respectively. 
	Observe that $D_{i}=C_{\psi _{i}}$, for $i=1,2$.
	From Theorem 4 in \cite{DS-2003}, we know that
	every Archimedean copula is Schur-concave,
	then $D_{1}$ is Schur-concave. Now, note that $C_{\psi _{1}}\leq C_{\psi _{2}}$
	is equivalent to $D_{1}\leq D_{2}$,
	then from Theorem \ref{th01}, the result is obtained.
\end{proof}

\begin{remark}
	Observe that  $C_{\psi _{1}}\leq C_{\psi _{2}}$ it is equivalent to
	$\phi_1\circ \psi _{2}$ subadditive (see Theorem 4.4.2 in \cite{N06}).
	In addition, from Lemma A.1 of \cite{LF2015}, we know that
	$C_{\psi _{1}}\leq C_{\psi _{2}}$
	 if $\phi _{2}\circ \psi _{1}$ is
	superadditive.
\end{remark}

\bigskip

Very recently, \cite{Barmalzan-2020} investigated stochastic comparisons of
series and parallel systems with dependent heterogeneous extended
exponential components under Archimedean copula with $F\overset{ST}{=}G$.
In this case, the distribution function of the $i$th extended exponential
component is%
\begin{equation}\label{eq_ee}
F(t,\lambda _{i})=\frac{1-e^{-\lambda _{i}t}}{1-\bar{\alpha }%
	e^{-\lambda _{i}t}},\quad\text{ for }t>0,
\end{equation}%
where $\alpha >0$, $\bar{\alpha }=1-\alpha $ and $\lambda _{i}>0$. 
Observe that \eqref{eq_ee} belongs to the MPHR model defined in \eqref{MPHR} 
where $H(t)=e^{-t}$. Then, from Lemma \ref{lema01}(iii), we know that 
$F(t,\lambda _{i})$ is increasing and concave in $\lambda_{i}$. 
Therefore, from Corollary \ref{cor01},
if $\phi _{2}\circ \psi _{1}$ is superadditive, we get that 
$\bm{\lambda} \overset{w}{\geq }\bm{\lambda }^{\ast }$
implies $X_{n:n}\geq _{ST}Y_{n:n}$.
In other words,
Theorem 1 in \cite{Barmalzan-2020}  can be seen as a particular case of Corollary \ref{cor01} here.
Observe also that, from our result, we know that $X_{n:n}\geq _{ST}Y_{n:n}$
holds for every $\alpha >0$, not only for $0<\alpha \leq 1$
as it is claimed in \cite{Barmalzan-2020}.

Next, we present another result to compare parallel systems with dependent components connected by Archimedean copulas.

\begin{proposition}\label{prop_01}
		Let $X_{1},\ldots ,X_{n}$ and $Y_{1},\ldots ,Y_{n}$ be two
	sets of nonnegative random variables with 
	$X_{i}\sim F\left( t,\lambda
	_{i}\right) $ and 
	$Y_{i}\sim G\left( t,\lambda _{i}^{\ast }\right) $ for $%
	i=1,\ldots ,n$, with associated Archimedean copulas $C_{\psi _{1}}$ and $%
	C_{\psi _{2}}$, respectively. 
	Suppose that $C_{\psi _{1}}\leq C_{\psi _{2}}$
 and $F\geq _{ST}G$.
 If $\phi _{1}(F(t,\lambda))$ is convex in $\lambda$
		and $F\left( t,\lambda\right) $ is
		increasing (decreasing) in $\lambda$,
		then ${\bm\lambda} 
		\overset{w}{\geq }{\bm\lambda }^{\ast }\left( {\bm\lambda \geq }_{w}{%
			\bm\lambda }^{\ast }\right) $ implies $X_{n:n}\geq _{ST}Y_{n:n}$.
\end{proposition}

\begin{proof} 
	The distribution function of $X_{n:n}$ is%
	\begin{equation}\label{dist_max}
	F_{X_{n:n}}(t)=\psi _{1}\left( \sum_{i=1}^{n}\phi _{1}\left( F\left(
	t,\lambda _{i}\right) \right) \right) .
	\end{equation}%
	Let us denote $g(\lambda_i)=\phi _{1}(F(t,\lambda_i))$, then the above 
	distribution function can be rewritten as
	\[F_{X_{n:n}}(t)=\psi _{1}\left(\sum_{i=1}^{n} g(\lambda_i) \right).\]	
	Let $X_{1}^{\ast },\ldots ,X_{n}^{\ast }$ be a set of nonnegative random
	variables with $X_{i}^{\ast }\sim F\left( t,\lambda _{i}^{\ast }\right) $
	for $i=1,\ldots ,n$, and the same associated Archimedean copula than $%
	X_{1},\ldots ,X_{n}$. 
	Then, the distribution function of $X_{n:n}^{\ast }$ is%
	\begin{equation*}
	F_{X_{n:n}^{\ast }}(t)=\psi _{1}\left( \sum_{i=1}^{n}\phi _{1}\left( F\left(
	t,\lambda _{i}^{\ast }\right) \right) \right) 
	=\psi _{1}\left(\sum_{i=1}^{n} g(\lambda_i^{\ast }) \right).
	\end{equation*}%
	First, we need to prove that 
	$X_{n:n}\geq _{ST}X_{n:n}^{\ast }$ if 
	${\bm\lambda} \overset{w}{\geq }{\bm\lambda }^{\ast }\left( {\bm\lambda \geq }_{w}{\bm\lambda }^{\ast }\right) $. 
	It is evident that
	\[\frac{\partial g(\lambda_i)}{\partial \lambda_i}=
	\phi' _{1}(F(t,\lambda_i))\frac{\partial F(t,\lambda_i)}{\partial \lambda_i}
	\overset{sign}{=}-\frac{\partial F(t,\lambda_i)}{\partial \lambda_i}.
\]
Then, from Proposition 3.C.1.b in \cite{Marshall:2011}, we have that
$\sum_{i=1}^{n} g(\lambda_i)$ 
is decreasing (increasing) and Shur-convex when 
$F(t,\lambda_i)$ is increasing (decreasing) in $\lambda_i$.
Now, as $\psi _{1}$ is decreasing, we get
$\psi _{1}\left(\sum_{i=1}^{n} g(\lambda_i) \right)\leq \psi _{1}\left(\sum_{i=1}^{n} g(\lambda_i^{\ast }) \right)$. 
In other words, $X_{n:n}\geq
_{ST}X_{n:n}^{\ast }$ if ${\bm\lambda} \overset{w}{\geq }{\bm\lambda }^{\ast
}\left( {\bm\lambda \geq }_{w}{\bm\lambda }^{\ast }\right) $.

	On the other hand, let $Y_{1}^{\ast },\ldots ,Y_{n}^{\ast }$ be a set of
	nonnegative random variables with $Y_{i}^{\ast }\sim F\left( t,\lambda
	_{i}^{\ast }\right) $ for $i=1,\ldots ,n$, and the same associated
	Archimedean copula than $Y_{1},\ldots ,Y_{n}$. Then, the distribution
	function of $Y_{n:n}^{\ast }$ is
	
	\begin{equation*}
	F_{Y_{n:n}^{\ast }}(t)=\psi _{2}\left( \sum_{i=1}^{n}\phi _{2}\left( F\left(
	t,\lambda _{i}^{\ast }\right) \right) \right) =C_{\psi _{2}}\left( F\left(
	t,\lambda _{1}^{\ast }\right) ,\ldots ,F\left( t,\lambda _{n}^{\ast }\right)
	\right) .
	\end{equation*}%
	Now, as $C_{\psi _{1}}\leq C_{\psi _{2}}$, then $F_{X_{n:n}^{\ast
	}}\leq F_{Y_{n:n}^{\ast }}$. In other words, $X_{n:n}^{\ast }\geq
	_{ST}Y_{n:n}^{\ast }$. Finally, we know, from Proposition 2.3(i) in
	\cite{NASS16}, that 
	\[C_{\psi _{2}}\left( F\left( t,\lambda _{1}^{\ast
	}\right) ,\ldots ,F\left( t,\lambda _{n}^{\ast }\right) \right) \leq C_{\psi
		_{2}}\left( G\left( t,\lambda _{1}^{\ast }\right) ,\ldots ,G\left( t,\lambda
	_{n}^{\ast }\right) \right)\]
	 since $F\geq _{ST}G$. This complete the proof.
\end{proof}

\begin{remark}
	Taking the second partial derivative of $g=\phi _{1}\circ F$
	with respect to $\lambda$, we have
	\[\frac{\partial^{2}g(\lambda)}{\partial\lambda^{2}}=
	\phi'' _{1}(F(t,\lambda))\left(\frac{\partial F(t,\lambda) }{\partial\lambda}\right)^{2}+
	\phi' _{1}(F(t,\lambda))
	\frac{\partial^{2} F(t,\lambda) }{\partial\lambda^{2}}.
	\]
	Therefore, if $F$ is concave in $\lambda$, then $g$ is convex since
	$\phi _{1}$ is convex.
	In other words, the condition $\phi _{1}(F(t,\lambda))$  convex in $\lambda$ in Proposition \ref{prop_01} is a weaker condition than the condition
	$F$ concave with respect to $\lambda$ in Corollary \ref{cor01}.
\end{remark}

The following example provides an illustration of the theoretical result of
Proposition \ref{prop_01}.

\begin{example}\label{ex_max}
		We consider two parallel systems $X_{2:2}$ and $Y_{2:2}$
	with dependent components assembled by an Ali-Mikhail-Haq (AMH) copula 
	with generator $\psi(x)=(1-\theta)/(e^{x}-\theta)$  for $\theta\in(0,1)$.
	For simplicity, we assume $D_{1}=D_{2}=D$.
	Let us assume that the components have SPRH distribution as defined in
	\eqref{SPRH}.
	From Lemma \ref{lema01}(iv), we know that $F$ is convex with respect to $\alpha$, so we can not use Corollary \ref{cor01} in this case.
	However, the function
	\[g(\alpha)=\phi _{1}(F(t,\alpha))=\log\left(\frac{1-\theta}{H^{\alpha}(\lambda t)}+\theta\right)
	\]
	is convex in $\alpha$ for  $\theta\in(0,1)$ since
	\[\frac{\partial^2 g(\alpha)}{\partial \alpha^2}=
	\frac{\theta (1-\theta)   H^{\alpha}(\lambda t) \left( \log H^{\alpha}(\lambda t)\right)^2}{\left(1+\theta  \left( H^{\alpha}(\lambda t)-1\right)\right)^2}\geq 0.
	\]
	In addition, again from Lemma \ref{lema01}(iv), we know that $F$ is
	decreasing in $\alpha$. 
	{Assume that $H(t)=1-e^{-1.5 t^2}$ for $X_i$ and $H(t)=1-e^{-2.8 t^2}$ for $Y_i$, with $i=1,2$. Note that $F\geq_{ST}G$.}	
	Let us consider $\lambda=0.5$, $\bm\alpha=(0.3, 1.3)$ and
	$\bm\alpha^{*}=(0.5, 0.6)$. It is easy to verify that
	 ${\bm\alpha \geq }_{w}{\bm\alpha }^{\ast }$
	 and therefore, all the conditions in Proposition \ref{prop_01} hold.
	Figure \ref{Fig_maxx} shows the distribution functions of the two parallel systems for {$\theta=0.2$ (left) and $\theta=0.8$ (right)}.
	It can be seen that $X_{2:2}\geq _{ST}Y_{2:2}$
	according to Proposition \ref{prop_01}.	 
\end{example}

\begin{figure}[H]
	\centering
	\includegraphics[width=6 cm]{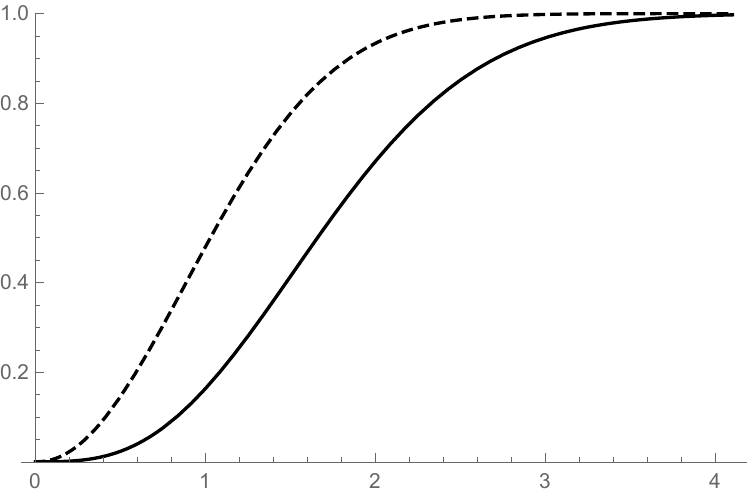}
	\includegraphics[width=6 cm]{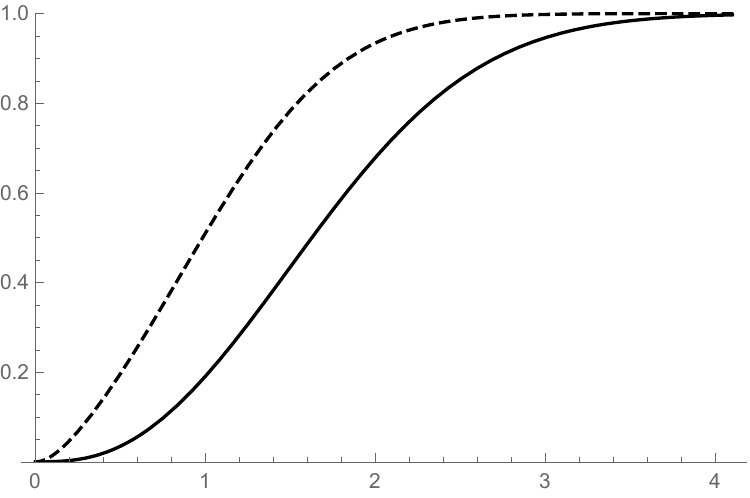}
	\caption{Plots of the distribution functions of $X_{2:2}$ (solid line)
		and $Y_{2:2}$  (dashed line)
		considered in Example \ref{ex_max}
		for dependent components assembled by AMH copulas
		when $\theta=0.2$  (left) and $\theta=0.8$ (right).}
	\label{Fig_maxx}
\end{figure}

\begin{remark}\label{rem_01}
	Analogously to Proposition \ref{prop_01}, it can be proved that
	if $C_{\psi _{1}}\geq C_{\psi _{2}}$, $F\leq _{ST}G$,
	$\phi _{1}(F(t,\lambda))$ concave in $\lambda$
	and $F\left( t,\lambda\right) $
	decreasing (increasing) in $\lambda$,
	then ${\bm\lambda} 
	\overset{w}{\geq }{\bm\lambda }^{\ast }\left( {\bm\lambda \geq }_{w}{%
		\bm\lambda }^{\ast }\right) $ implies $X_{n:n}\leq _{ST}Y_{n:n}$.
\end{remark}

\begin{proposition} \label{prop_03}
	Let $X_{1},\ldots ,X_{n}$ and $Y_{1},\ldots ,Y_{n}$ be two
	sets of nonnegative random variables with 
	$X_{i}\sim F\left( t,\lambda
	_{i}\right) $ and 
	$Y_{i}\sim G\left( t,\lambda _{i}^{\ast }\right) $ for $%
	i=1,\ldots ,n$, with associated Archimedean copulas $C_{\psi _{1}}$ and $%
	C_{\psi _{2}}$, respectively. 
	Suppose that $C_{\psi _{1}}\leq C_{\psi _{2}}$ and $F\geq _{ST}G$.
	If $F\left( t,\lambda\right)$ is increasing (decreasing) and log-concave in $\lambda$ and  
	$\psi _{1}$ is log-convex, then 
	${\bm\lambda} 
	\overset{w}{\geq }{\bm\lambda }^{\ast }\left( {\bm\lambda \geq }_{w}{%
		\bm\lambda }^{\ast }\right) $ implies $X_{n:n}\geq _{ST}Y_{n:n}$.
\end{proposition}

\begin{proof} 
	The distribution function of $X_{n:n}$ is as defined in \eqref{dist_max}.
	Analogously, we define the distribution function of $X^{*}_{n:n}$ by changing $\lambda$ by $\lambda^{*}$.
	First, we need to prove that 
	$X_{n:n}\geq _{ST}X_{n:n}^{\ast }$ if 
	${\bm\lambda} \overset{w}{\geq }{\bm\lambda }^{\ast }\left( {\bm\lambda \geq }_{w}{\bm\lambda }^{\ast }\right) $. 
	Let us denote 
	$\varphi \left( \lambda
	_{1},\ldots ,\lambda _{n}\right) =\psi _{1}\left( \sum_{i=1}^{n}\phi
	_{1}\left( F\left( t,\lambda _{i}\right) \right) \right) $. 
	Taking the
	derivative of $\varphi \left( {\bm\lambda }\right) $
	with respect to 
	$\lambda _{i}$, we get%
	\begin{eqnarray*}
		\frac{\partial \varphi \left( {\bm\lambda }\right) }{\partial \lambda
			_{i}} &=&\psi _{1}^{\prime }\left( \sum_{i=1}^{n}\phi _{1}\left( F\left(
		t,\lambda _{i}\right) \right) \right) \phi _{1}^{\prime }\left( F\left(
		t,\lambda _{i}\right) \right) \frac{\partial F\left( t,\lambda _{i}\right) }{%
			\partial \lambda _{i}} \\
		&=&\psi _{1}^{\prime }\left( \sum_{i=1}^{n}\phi _{1}\left( F\left( t,\lambda
		_{i}\right) \right) \right) \frac{1}{\psi _{1}^{\prime }\left( \phi
			_{1}\left( F\left( t,\lambda _{i}\right) \right) \right) }\frac{\partial
			F\left( t,\lambda _{i}\right) }{\partial \lambda _{i}}\\
		&\overset{sign}{=}&\frac{\partial
			F\left( t,\lambda _{i}\right) }{\partial \lambda _{i}}.	
	\end{eqnarray*}%
 For any pair 
	$1\leq i\leq j\leq n$, we have%
	\begin{eqnarray}\label{eq_fi}\nonumber
		\left( \lambda _{i}-\lambda _{j}\right) \left( \frac{\partial \varphi
			\left( {\bm\lambda }\right) }{\partial \lambda _{i}}-\frac{\partial
			\varphi \left( {\bm\lambda }\right) }{\partial \lambda _{j}}\right)
		&=&\left( \lambda _{i}-\lambda _{j}\right) \psi _{1}^{\prime }\left(
		\sum_{i=1}^{n}\phi _{1}\left( F\left( t,\lambda _{i}\right) \right) \right)\\\nonumber
		&&
		\left( \frac{1}{\psi _{1}^{\prime }\left( \phi _{1}\left( F\left( t,\lambda
			_{i}\right) \right) \right) }\frac{\partial F\left( t,\lambda _{i}\right) }{%
			\partial \lambda _{i}}-\frac{1}{\psi _{1}^{\prime }\left( \phi _{1}\left(
			F\left( t,\lambda _{j}\right) \right) \right) }\frac{\partial F\left(
			t,\lambda _{j}\right) }{\partial \lambda _{j}}\right) \\
		&\overset{sign}{=}&
		-\left( \lambda _{i}-\lambda _{j}\right) \left( \frac{%
			\partial _{\lambda _{i}}F\left( t,\lambda _{i}\right) }{\psi _{1}^{\prime
			}\left( \phi _{1}\left( F\left( t,\lambda _{i}\right) \right) \right) }-%
		\frac{\partial _{\lambda _{j}}F\left( t,\lambda _{j}\right) }{\psi
			_{1}^{\prime }\left( \phi _{1}\left( F\left( t,\lambda _{j}\right) \right)
			\right) }\right)
	\end{eqnarray}%
	since, $\psi _{1}^{\prime }\leq 0$. Now, let us denote
	\begin{eqnarray}\label{eq_w}
	\ell(\lambda)=\frac{\partial_{\lambda}F\left( t,\lambda\right)}{F\left( t,\lambda\right)}
	\cdot
	\frac{\psi
		_{1}\left( \phi _{1}\left( F\left( t,\lambda\right) \right)
		\right)}{\psi
		_{1}^{\prime }\left( \phi _{1}\left( F\left( t,\lambda\right) \right)
		\right)}
	=\ell_1(\lambda)\cdot \ell_2(\lambda),
	\end{eqnarray}
	where
	\[\ell_1(\lambda)=\frac{\partial\log  F\left( t,\lambda\right)}{\partial \lambda} 
	\quad\textup{and}\quad
	\ell_2(\lambda)=\frac{\psi_{1}}{\psi
		_{1}^{\prime }}
	\Big( \phi _{1}\left( F\left( t,\lambda\right) \right)
	\Big).
	\]
	Then, \eqref{eq_fi} can be rewritten as
	\[\left( \lambda _{i}-\lambda _{j}\right) \left( \frac{\partial \varphi
		\left( {\bm\lambda }\right) }{\partial \lambda _{i}}-\frac{\partial
		\varphi \left( {\bm\lambda }\right) }{\partial \lambda _{j}}\right)
	\overset{sign}{=}
	-\left( \lambda _{i}-\lambda _{j}\right)
	\Big(\ell(\lambda_{i})-\ell(\lambda_{j})	\Big),	\]
	since $\psi
	_{1}\left( \phi _{1}\left( F\left( t,\lambda\right) \right)
	\right)= F\left( t,\lambda\right)$.
	Therefore, we need to prove that $\ell$ is increasing.
	
	Note that the function $\ell_2$ is negative-valued since $\psi
	_{1}^{\prime }\leq 0$.	
	In addition, observe that if $\psi _{1}/\psi _{1}^{\prime }$
	is decreasing  in $t$, which is equivalent to $\psi _{1}$ log-convex, 
	and $F$ is increasing (decreasing), then $\ell_2$ 
	is increasing (decreasing).
	On the other hand, it is evident that $\ell_1$ is decreasing since $F$ is log-concave in $\lambda$ and it is positive-valued (negative-valued) if 
	 $F$ is increasing (decreasing) in $\lambda$.
	Therefore,
	from \eqref{eq_w}, we have that $\ell$
 is increasing when $F$ is increasing (decreasing) and log-concave in $\lambda$, and $\psi _{1}$ log-convex.
 Thus, from Theorem 3.A.4 in \cite{Marshall:2011}, we know that
	the function $\varphi \left( {\bm\lambda }\right) $ is Schur-concave
	with respect to ${\bm\lambda }$. 
	In other words, $X_{n:n}\geq
	_{ST}X_{n:n}^{\ast }$ if ${\bm\lambda} \overset{w}{\geq }{\bm\lambda }^{\ast
	}\left( {\bm\lambda \geq }_{w}{\bm\lambda }^{\ast }\right) $.
The rest
of the proof is the same than in Proposition \ref{prop_01}, hence omitted.
\end{proof} 

\bigskip

The next example illustrates the result of Proposition \ref{prop_03}.

\begin{example}\label{ex_logconcave}
	Let us consider two parallel systems $X_{3:3}$ and $X^{*}_{3:3}$ with $F\overset{ST}{=}G$, 	for simplicity. 
	We assume that $F$ belongs to the TIIHL-RS family
	defined in \eqref{eq_tiil},
	where $H$ is an exponential distribution with hazard rate $2$
	 and 
	${C}_{\psi _{i}}$ is a Clayton copula with generator
	$\psi _{i}=(\theta_i t +1)^{-1/\theta_i}$ with $\theta_i\geq 0$
	for $i=1,2$.
	Then, their distribution functions are
	\begin{equation*}
	{F}_{X_{3:3}}(t)=\left[ \left( \frac{2 \left(1-e^{-2 t}\right)^{\alpha_1}}{1+\left(1-e^{-2 t}\right)^{\alpha_1}}\right) ^{-\theta _{1}}+\left( \frac{2 \left(1-e^{-2 t}\right)^{\alpha_2}}{1+\left(1-e^{-2 t}\right)^{\alpha_2}}\right)
	^{-\theta _{1}}
	+\left( \frac{2 \left(1-e^{-2 t}\right)^{\alpha_3}}{1+\left(1-e^{-2 t}\right)^{\alpha_3}}\right)
	^{-\theta _{1}}
	-2\right] ^{-1/\theta _{1}}
	\end{equation*}%
	and%
	\begin{equation*}
	{F}_{X_{3:3}^{\ast }}(t)=\left[ \left( \frac{2 \left(1-e^{-2 t}\right)^{\alpha_1^{\ast }}}{1+\left(1-e^{-2 t}\right)^{\alpha_1^{\ast }}}\right) ^{-\theta _{1}}+\left( \frac{2 \left(1-e^{-2 t}\right)^{\alpha_2^{\ast }}}{1+\left(1-e^{-2 t}\right)^{\alpha_2^{\ast }}}\right)
	^{-\theta _{1}}
	+\left( \frac{2 \left(1-e^{-2 t}\right)^{\alpha_3^{\ast }}}{1+\left(1-e^{-2 t}\right)^{\alpha_3^{\ast }}}\right)
	^{-\theta _{1}}
	-2\right]
	^{-1/\theta _{2}},
	\end{equation*}
	respectively.
	From Lemma \ref{lema01}(i), we know that $F$ is convex in $\alpha$, so
	we can not apply Corollary \ref{cor01}. However $F$ is log-concave and decreasing in $\alpha$ for any $\lambda>0$, so that we can use Proposition \ref{prop_03}.		
	Also notice that $\log \psi _{1}(t)=-(1/\theta_1)\log(\theta_1 t +1) $ and it is easy to check that $\psi _{1}$ is log-convex.
	
	On the other hand, we have
	$\phi _{2}\circ \psi _{1}(t)=\theta^{-1}_{2}(\theta_1 t +1)^{\theta_{2}/\theta_{1}}$.
	Taking the derivative of $\phi _{2}\circ \psi _{1}$
	twice with respect to $t$, it can be seen that
	$[\phi _{2}\circ \psi _{1}(t)]''\overset{sign}{=}
	\theta_{2}-\theta_{1}$. Therefore, 
	$\phi _{2}\circ \psi _{1}$ is convex for $\theta_{2}>\theta_{1}$
	and concave for $\theta_{2}<\theta_{1}$. In other words, 
	$C_{\psi _{1}}\leq(\geq) C_{\psi _{2}}$ if $\theta_{1}<(>)\theta_{2}$.
		Finally, we take 
	${\bm\alpha}=(0.6, 1.1, 2.5)$ and ${\bm\alpha}^{\ast }=(0.4, 1.4, 2.1)$.
	It is then easy to observe that
	${\bm\alpha}{\geq_{w}}{\bm\alpha }^{\ast }$.
	Figure \ref{Fig_logconcave} shows the distribution functions of 
	$X_{3:3}$ (solid line) and $X_{3:3}^{*}$ (dashed line)
	for $\theta_1=2$ and $\theta_2=15$ (left) and 
		for $\theta_1=20$ and $\theta_2=1.5$ (right).
	As it is seen in Figure \ref{Fig_logconcave}(left), 
	$X_{3:3}\geq_{ST}X_{3:3}^{*}$, which coincides with the result of Proposition \ref{prop_03}, since all the conditions  
	in that proposition are satisfied.
	However, in Figure \ref{Fig_logconcave}(right), the distribution functions cross each other
	and this means that $X_{3:3}$ and $X^{*}_{3:3}$ are not ordered
	since $\phi _{2}\circ \psi _{1}$ is not
	superadditive (i.e. $C_{\psi _{1}}\nleq C_{\psi _{2}}$).
\end{example}

\begin{figure}[H]
	\centering
	\includegraphics[width=6 cm]{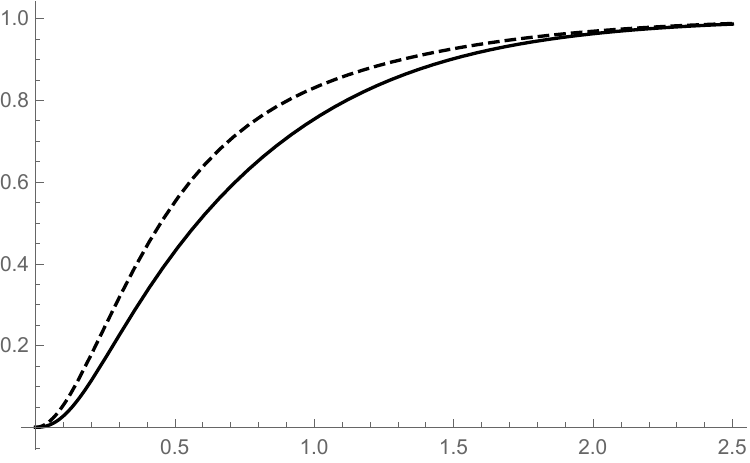}
	\includegraphics[width=6 cm]{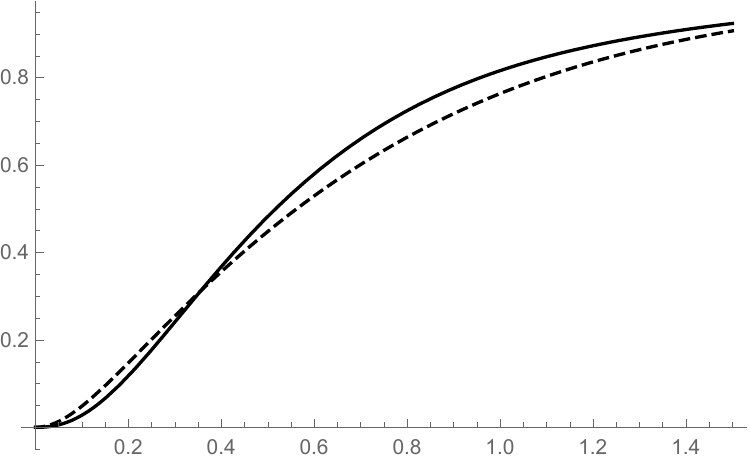}
	\caption{Plots of the distributions functions for  $X_{3:3}$ (solid line) and $X_{3:3}^{*}$ (dashed line)
		for $\theta_1=2$ and $\theta_2=15$ (left) and 
		for $\theta_1=20$ and $\theta_2=1.5$ (right) considered in Example \protect\ref{ex_logconcave}.}
	\label{Fig_logconcave}
\end{figure}

{From Lemma \ref{lema01}, we know that the TIIHL-RS model for $\alpha>0$, the MPRHR model for $\alpha>1$,
	the SPRH model for $\lambda>0$ when $\tilde{r}$ decreasing, and the EL-S model for $\lambda>0$ when $\tilde{r}$ decreasing
	verify the conditions on $F$ in Proposition \ref{prop_03}. Therefore, we can obtain comparisons between parallel systems following
	these type of distributions.}

\begin{remark}\label{rem_03}
	It is worth mentioning that,
		from \eqref{eq_w}, we have that the function $\ell$
	is decreasing when $F$ is decreasing (increasing) and log-convex in $\lambda$ and $\psi _{1}$ log-concave.
	Thus,
	the function $\varphi \left( {\bm\lambda }\right) $ is Schur-convex
	with respect to ${\bm\lambda }$. 
	Then, if $C_{\psi _{1}}\geq C_{\psi _{2}}$ and $F\leq _{ST}G$,	
	the condition
	 ${\bm\lambda} \overset{w}{\geq }{\bm\lambda }^{\ast
	}\left( {\bm\lambda \geq }_{w}{\bm\lambda }^{\ast }\right) $
implies
 $X_{n:n}\leq
	_{ST}Y_{n:n}$.
\end{remark}

{Analogously, again from Lemma \ref{lema01}, we know that the MPRHR model for $0 <\alpha\leq 1$,
	the SPRH model for $\lambda>0$ when $\tilde{r}$ increasing, and the EL-S model for $\theta>0$ when $x\tilde{r}(x)$ increasing in $x$
	verify the conditions on $F$ in Remark \ref{rem_03}.
	Thus, we can also obtain comparisons between parallel systems following
	these type of models.
The condition $x\tilde{r}\left( x\right)$ increasing (decreasing)
and other related aging notions are investigated in \cite{OlT-15}.
For distributions with bounded support and
bathtub shaped hazard function, see for instance \cite{LJX16}.}


\subsection{Series systems with Archimedean copulas}

Presently, let us consider two series systems $X_{1:n}$ and $Y_{1:n}$ with dependent
heterogeneous components assembled by Archimedean survival copulas $%
\widehat{C}_{\psi _{1}}$ and $\widehat{C}_{\psi _{2}}$, respectively. Then, their distributions
functions are 
\[F_{X_{1:n}}(t)=D_{1}\left( F\left( t,\lambda _{1}\right) ,\ldots
,F\left( t,\lambda _{n}\right) \right) =1-\widehat{C}_{\psi _{1}}\left(
1-F\left(t, \lambda _{1}\right) ,\ldots ,1-F\left(t, \lambda _{n}\right) \right) 
\]
and 
\[F_{Y_{1:n}}(t)=D_{2}\left( G\left(t, \lambda _{1}^{\ast }\right) ,\ldots
,G\left( t,\lambda _{n}^{\ast }\right) \right) =1-\widehat{C}_{\psi
	_{2}}\left( 1-G\left(t, \lambda _{1}^{\ast }\right) ,\ldots ,1-G\left(t, \lambda
_{n}^{\ast }\right) \right), \] 
respectively. As before, if $\phi _{2}\circ
\psi _{1}$ is superadditive then $\widehat{C}_{\psi _{1}}\left( u\right)
\leq \widehat{C}_{\psi _{2}}\left( u\right) $ for all $u\in \left[ 0,1\right]
^{n}$, i.e. $D_{1}\geq D_{2}$. It is easy to check that $D_{1}$ is
Schur-convex, since any Archimedean copula is Schur-concave from Theorem 4 in \cite{DS-2003}. Then, from Theorem \ref{th02}, we can establish the following result.

\begin{corollary}\label{cor02} 
	Let $X_{1},\ldots ,X_{n}$ and $Y_{1},\ldots ,Y_{n}$ be two sets
	of nonnegative random variables with $X_{i}\sim F\left( t,\lambda
	_{i}\right) $ and $Y_{i}\sim G\left( t,\lambda _{i}^{\ast }\right) $ for $%
	i=1,\ldots ,n$, with associated Archimedean survival copulas $\widehat{C}_{\psi _{1}}$ and $%
	\widehat{C}_{\psi _{2}}$, respectively. Suppose that $\widehat C_{\psi _{1}}\leq \widehat C_{\psi _{2}}$, $F\leq _{ST}G$ and $F\left( t,\lambda\right) $ is
	decreasing (increasing) and convex in $\lambda$. If ${\bm\lambda} 
	\overset{w}{\geq }{\bm\lambda }^{\ast }\left( {\bm\lambda \geq }_{w}{%
		\bm\lambda }^{\ast }\right) $ then $X_{1:n}\leq _{ST}Y_{1:n}$.
\end{corollary} 

In the following result, we establish the usual stochastic order between the lifetimes of series systems when their components are
assembled by Archimedean copulas under different conditions than those considered in Corollary \ref{cor02} above.

\begin{proposition}\label{prop_02}
	Let $X_{1},\ldots ,X_{n}$ and $Y_{1},\ldots ,Y_{n}$ be two
	sets of nonnegative random variables with $X_{i}\sim F\left( t,\lambda
	_{i}\right) $ and $Y_{i}\sim G\left( t,\lambda _{i}^{\ast }\right) $ for $%
	i=1,\ldots ,n$, with associated Archimedean survival copulas $\widehat C_{\psi _{1}}$ and $%
	 \widehat C_{\psi _{2}}$, respectively. Suppose that $\widehat C_{\psi _{1}}\leq \widehat C_{\psi _{2}}$ and $F\leq _{ST}G$. 
	If $\phi _{1}(1-F(t,\lambda))$ is convex in $\lambda$
	and $F\left( t,\lambda\right) $ is
	decreasing (increasing) in $\lambda$,
	then ${\bm\lambda} 
	\overset{w}{\geq }{\bm\lambda }^{\ast }\left( {\bm\lambda \geq }_{w}{%
		\bm\lambda }^{\ast }\right) $ implies $X_{1:n}\leq _{ST}Y_{1:n}$.
\end{proposition}

\begin{proof}
	The distribution function of $X_{1:n}$ is%
	\begin{equation}\label{eq_min}
	F_{X_{1:n}}(t)=1-\psi _{1}\left( \sum_{i=1}^{n}\phi _{1}\left( 1-F\left(
	t,\lambda _{i}\right) \right) \right) = 
	1-\psi _{1}\left( \sum_{i=1}^{n} g(\lambda _{i})   \right),
	\end{equation}%
	where $g(\lambda_i)=\phi _{1}(1-F(t,\lambda_i))$.
	Let $X_{1}^{\ast },\ldots ,X_{n}^{\ast }$ be a set of nonnegative random
	variables with $X_{i}^{\ast }\sim F\left( t,\lambda _{i}^{\ast }\right) $
	for $i=1,\ldots ,n$, and the same associated Archimedean copula than $%
	X_{1},\ldots ,X_{n}$. Then, the distribution function of $X_{1:n}^{\ast }$ is%
	\begin{equation*}
	F_{X_{1:n}^{\ast }}(t)=1-\psi _{1}\left( \sum_{i=1}^{n}\phi _{1}\left(
	1-F\left( t,\lambda _{i}^{\ast }\right) \right) \right) 
	= 
	1-\psi _{1}\left( \sum_{i=1}^{n} g(\lambda _{i}^{*})   \right).
	\end{equation*}%
	First, we need to prove that $X_{1:n}\leq _{ST}X_{1:n}^{\ast }$ if ${%
		\bm\lambda} \overset{w}{\geq }{\bm\lambda }^{\ast }\left( {\bm\lambda \geq }_{w}%
	{\bm\lambda }^{\ast }\right) $. 
		It is evident that
	\[\frac{\partial g(\lambda_i)}{\partial \lambda_i}=-
	\phi' _{1}(1-F(t,\lambda_i))\frac{\partial F(t,\lambda_i)}{\partial \lambda_i}
	\overset{sign}{=}\frac{\partial F(t,\lambda_i)}{\partial \lambda_i}.
	\]
	Then, from Proposition 3.C.1.b in \cite{Marshall:2011}, we have that
	$\sum_{i=1}^{n} g(\lambda_i)$ 
	is decreasing (increasing) and Shur-convex when 
	$F(t,\lambda_i)$ is decreasing (increasing) in $\lambda_i$.
	Now, as $\psi _{1}$ is decreasing, we get
	$\psi _{1}\left(\sum_{i=1}^{n} g(\lambda_i) \right)\leq \psi _{1}\left(\sum_{i=1}^{n} g(\lambda_i^{\ast }) \right)$. 
	In other words, $X_{1:n}\leq
	_{ST}X_{1:n}^{\ast }$ if ${\bm\lambda} \overset{w}{\geq }{\bm\lambda }^{\ast
	}\left( {\bm\lambda \geq }_{w}{\bm\lambda }^{\ast }\right) $.

		On the other hand, let $Y_{1}^{\ast },\ldots ,Y_{n}^{\ast }$ be a set of
	nonnegative random variables with $Y_{i}^{\ast }\sim F\left( t,\lambda
	_{i}^{\ast }\right) $ for $i=1,\ldots ,n$, and the same associated
	Archimedean copula than $Y_{1},\ldots ,Y_{n}$. Then, the distribution
	function of $Y_{1:n}^{\ast }$ is
	
	\begin{equation*}
	F_{Y_{1:n}^{\ast }}(t)=1-\psi _{2}\left( \sum_{i=1}^{n}\phi _{2}\left(
	1-F\left( t,\lambda _{i}^{\ast }\right) \right) \right) =1-\widehat{C}_{\psi
		_{2}}\left( 1-F\left( t,\lambda _{1}^{\ast }\right) ,\ldots ,1-F\left(
	t,\lambda _{n}^{\ast }\right) \right) .
	\end{equation*}%
	Next, again as $\hat{C}_{\psi _{1}}\leq \hat{C}_{\psi _{2}}$, then $%
	F_{X_{n:n}^{\ast }}\geq F_{Y_{n:n}^{\ast }}$. In other words, $%
	X_{1:n}^{\ast }\leq _{ST}Y_{1:n}^{\ast }$. Finally, we know, from
	Proposition 2.3(i) in \cite{NASS16}, that 
	\[\widehat{C}_{\psi _{2}}\left(
	F\left( t,\lambda _{1}^{\ast }\right) ,\ldots ,F\left( t,\lambda _{n}^{\ast
	}\right) \right) \leq \widehat{C}_{\psi _{2}}\left( G\left( t,\lambda
	_{1}^{\ast }\right) ,\ldots ,G\left( t,\lambda _{n}^{\ast }\right) \right)\]
	since $F\leq _{ST}G$. This complete the proof.
\end{proof}

\begin{remark}
	Taking the second partial derivative of $g=\phi _{1}\circ \bar{F}$
	with respect to $\lambda$, we have
	\[\frac{\partial^{2}g(\lambda)}{\partial\lambda^{2}}=
	\phi'' _{1}(1-F(t,\lambda))\left(-\frac{\partial F(t,\lambda) }{\partial\lambda}\right)^{2}-
	\phi' _{1}(F(t,\lambda))
	\frac{\partial^{2} F(t,\lambda) }{\partial\lambda^{2}}.
	\]
	Therefore, if $F$ is convex in $\lambda$, then $g$ is convex since
	$\phi _{1}$ is convex.
	In other words, the condition $\phi _{1}(1-F(t,\lambda))$  convex in $\lambda$ in Proposition \ref{prop_02} is a weaker condition than the condition
	$F$ convex with respect to $\lambda$ in Corollary \ref{cor02}.
\end{remark}

\begin{remark}\label{rem_02}
	Analogously to Proposition \ref{prop_02}, it can be proved that
	if $\widehat C_{\psi _{1}}\geq \widehat C_{\psi _{2}}$, $F\geq _{ST}G$,
	$\phi _{1}(1-F(t,\lambda))$ concave in $\lambda$
	and $F\left( t,\lambda\right) $
	increasing (decreasing) in $\lambda$,
	then ${\bm\lambda} 
	\overset{w}{\geq }{\bm\lambda }^{\ast }\left( {\bm\lambda \geq }_{w}{%
		\bm\lambda }^{\ast }\right) $ implies $X_{1:n}\geq _{ST}Y_{1:n}$.
\end{remark}

To illustrate Proposition \ref{prop_02}, we consider the following numerical example.

\begin{example}\label{ex_min}
	Let us consider two series systems $X_{1:3}$ and $X^{*}_{1:3}$
	which dependent components having exponential distributions with hazard rate $\lambda>0$. It is easy to verify that $F(t,\lambda)=1-e^{-\lambda t}$
	is concave with respect to $\lambda$. Then,  we can not use 
	Corollary \ref{cor02} in order to compare the two series systems. We assume that the components are assembled by the same Clayton copula with generator $\psi(x)=(\theta x+1)^{-1/\theta}$ with $\theta>0$. Then, we have $\phi(x)=\theta^{-1}(x^{-\theta} -1)$, so that
	\[g(\lambda)=\phi(1-F(t,\lambda))=\theta^{-1}\Big((e^{-\lambda t})^{-\theta} -1\Big)	\]
	and then
	\[\frac{\partial^{2}g(\lambda)}{\partial\lambda^{2}}=
	\theta  t^2 \left(e^{-\lambda  t}\right)^{-\theta }\geq 0,	\]
	which means that $\phi(1-F(t,\lambda))$ is convex in $\lambda$.
	Set $\bm\lambda =(0.7, 0.8, 3.4)$ and $\bm\lambda^{*} =(0.9, 1.3, 2.5)$.
	Note that ${\bm\lambda \geq }_{w}{\bm\lambda }^{\ast } $.
	Figure \ref{Fig_minn} displays the distribution functions of 
	$X_{1:3}$ (solid line) and 
	$X_{1:3}^{*}$ (dashed line) for $\theta=0.6$ (left) and $\theta=6$ (right),
	from which we can see that $X_{1:3}\leq _{ST}X^{*}_{1:3}$
	as expected from Proposition \ref{prop_02}.	
\end{example}

\begin{figure}[H]
	\centering
	\includegraphics[width=6 cm]{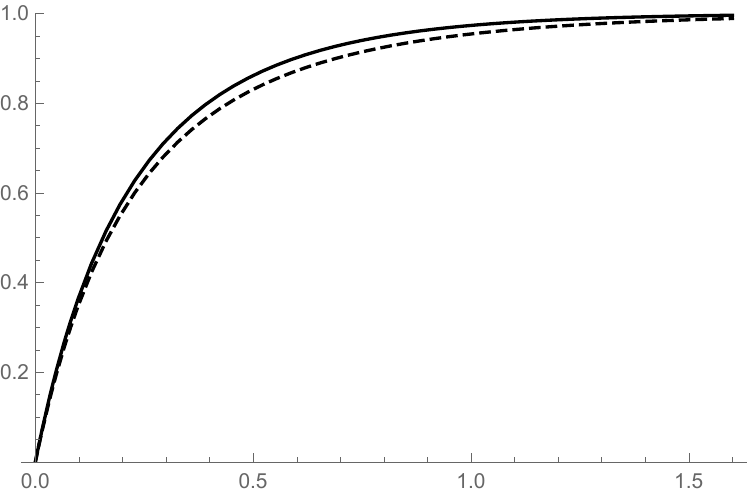}
		\includegraphics[width=6 cm]{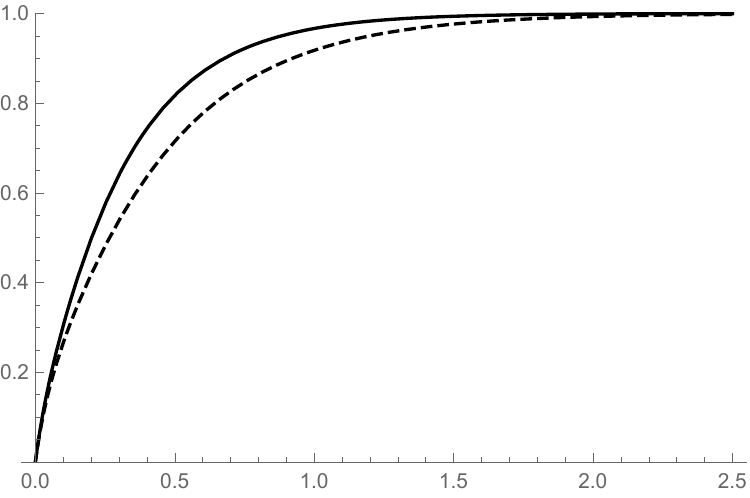}
	\caption{Plots of the distribution functions of  $X_{1:3}$ (solid line) and 
		$X_{1:3}^{*}$ (dashed line) consider in Example \ref{ex_min}
		for $\theta=0.6$ (left) and $\theta=6$ (right).}
	\label{Fig_minn}
\end{figure}

The following example shows that the condition  $\widehat C_{\psi _{1}}\leq \widehat C_{\psi _{2}}$ given in the
above Proposition \ref{prop_02} can not be dropped.

\begin{example}\label{ex_01}
	Let us consider two series systems $X_{1:2}$ and $X^{*}_{1:2}$ with $F\overset{ST}{=}G$, 	for simplicity. 
	We assume that $F$ is defined as in \eqref{eq_ee} and 
	$\widehat{C}_{\psi _{i}}$ is a Clayton copula with generator
	$\psi _{i}=(\theta_i t +1)^{-1/\theta_i}$ with $\theta_i\geq 0$
	for $i=1,2$.
	Then, their reliability functions are
	\begin{equation*}
	\bar{F}_{X_{1:2}}(t)=\left[ \left( \frac{\alpha e^{-\lambda _{1}t}}{1-%
		\overline{\alpha }e^{-\lambda _{1}t}}\right) ^{-\theta _{1}}+\left( \frac{%
		\alpha e^{-\lambda _{2}t}}{1-\overline{\alpha }e^{-\lambda _{2}t}}\right)
	^{-\theta _{1}}-1\right] ^{-1/\theta _{1}}
	\end{equation*}%
	and%
	\begin{equation*}
	\bar{F}_{X_{1:2}^{\ast }}(t)=\left[ \left( \frac{\alpha e^{-\lambda
			_{1}^{\ast }t}}{1-\overline{\alpha }e^{-\lambda _{1}^{\ast }t}}\right)
	^{-\theta _{2}}+\left( \frac{\alpha e^{-\lambda _{2}^{\ast }t}}{1-\overline{%
			\alpha }e^{-\lambda _{2}^{\ast }t}}\right) ^{-\theta _{2}}-1\right]
	^{-1/\theta _{2}},
	\end{equation*}
	respectively.
	Observe that
	\[g(\lambda)=\phi(\bar F(t,\lambda))=\theta^{-1}\left(
	\left(\frac{1-\bar{\alpha} e^{-\lambda  t}}{\alpha  e^{-\lambda  t}}\right)^{\theta }-1\right).	\]
	Taking the second partial derivative of the above function with respect to $\lambda$, we get
	\[\frac{\partial^{2}g(\lambda)}{\partial\lambda^{2}}=
	\frac{t^2 e^{\lambda  t} \left(\theta  e^{\lambda  t}-\bar{\alpha} \right) \left(e^{\lambda  t}-\bar{\alpha} \right)^{\theta }}{\alpha^{\theta}\left(\bar{\alpha} -e^{\lambda  t}\right)^2}
	\overset{sign}{=}\theta  e^{\lambda  t}+\alpha-1\geq 0
	\]
	 when $\alpha>1$ or $\theta>1$ and therefore, $\phi(\bar F(t,\lambda))$  is convex in that case.

	On the other hand,
	as we show in Example \ref{ex_logconcave},
	$\widehat C_{\psi _{1}}\leq(\geq) \widehat C_{\psi _{2}}$ if $\theta_{1}<(>)\theta_{2}$.
	Let us consider $\theta_{1}=5$, $\theta_{2}=1.2$ and $\alpha=0.4$,
	so that  $\phi(\bar F(t,\lambda))$  is convex but $\phi _{2}\circ \psi _{1}$ is not
	superadditive (i.e., $\widehat C_{\psi _{1}}\geq \widehat C_{\psi _{2}}$).
	Finally, we take 
	${\bm\lambda}=(2, 4)$ and ${\bm\lambda}^{\ast }=(3, 3)$.
	It is then easy to observe that
	${\bm\lambda} \overset{m}{\geq}{\bm\lambda }^{\ast }$.
	As it is seen in Figure \ref{Fig_01}, the reliability functions cross each other
	and this means that $X_{1:2}$ and $X^{*}_{1:2}$ are not ordered
	since $\phi _{2}\circ \psi _{1}$ is not
	superadditive.
\end{example}

\begin{figure}[H]
	\centering
	\includegraphics[width=6 cm]{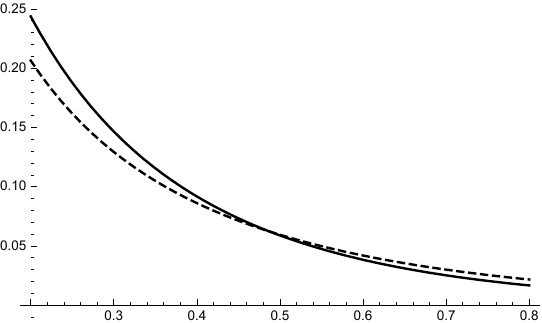}
	\caption{Plots of the reliability functions for  $X_{1:2}$ and $X^{*}_{1:2}$ consider in Example \protect\ref{ex_01}.}
	\label{Fig_01}
\end{figure}

\begin{proposition} \label{prop_04}
	Let $X_{1},\ldots ,X_{n}$ and $Y_{1},\ldots ,Y_{n}$ be two
	sets of nonnegative random variables with 
	$X_{i}\sim F\left( t,\lambda
	_{i}\right) $ and 
	$Y_{i}\sim G\left( t,\lambda _{i}^{\ast }\right) $ for $%
	i=1,\ldots ,n$, with associated Archimedean survival copulas $\widehat C_{\psi _{1}}$ and $%
	\widehat C_{\psi _{2}}$, respectively. 
	Suppose that $\widehat C_{\psi _{1}}\leq \widehat C_{\psi _{2}}$ and $F\geq _{ST}G$.
	If $F\left( t,\lambda\right)$ is decreasing (increasing) in $\lambda$, $\bar F$ log-concave in $\lambda$ and  
	$\psi _{1}$ log-convex, then 	
	${\bm\lambda} 
	\overset{w}{\geq }{\bm\lambda }^{\ast }\left( {\bm\lambda \geq }_{w}{%
		\bm\lambda }^{\ast }\right) $ implies $X_{1:n}\leq _{ST}Y_{1:n}$.
\end{proposition}

\begin{proof}
	The distribution function of $X_{1:n}$ is the same than in \eqref{eq_min}.
	Analogously, the distribution function of $X^{*}_{1:n}$ can be defined
	by changing $\lambda$ by $\lambda^{*}$.
	Firstly, we need to prove that $X_{1:n}\leq _{ST}X_{1:n}^{\ast }$ if ${%
		\bm\lambda} \overset{w}{\geq }{\bm\lambda }^{\ast }\left( {\bm\lambda \geq }_{w}%
	{\bm\lambda }^{\ast }\right) $. 
		Let us denote $\varphi \left( \lambda
	_{1},\ldots ,\lambda _{n}\right) =1-\psi _{1}\left( \sum_{i=1}^{n}\phi
	_{1}\left( 1-F\left( t,\lambda _{i}\right) \right) \right) $. Taking the
	derivative of $\varphi \left( {\bm\lambda }\right) $ with respect to $%
	\lambda _{i}$, we get%
	\begin{eqnarray*}
		\frac{\partial \varphi \left( {\bm\lambda }\right) }{\partial \lambda
			_{i}} &=&\psi _{1}^{\prime }\left( \sum_{i=1}^{n}\phi _{1}\left( 1-F\left(
		t,\lambda _{i}\right) \right) \right) \phi _{1}^{\prime }\left( 1-F\left(
		t,\lambda _{i}\right) \right) \frac{\partial F\left( t,\lambda _{i}\right) }{%
			\partial \lambda _{i}} \\
		&=&\psi _{1}^{\prime }\left( \sum_{i=1}^{n}\phi _{1}\left( 1-F\left(
		t,\lambda _{i}\right) \right) \right) \frac{1}{\psi _{1}^{\prime }\left(
			\phi _{1}\left( 1-F\left( t,\lambda _{i}\right) \right) \right) }\frac{%
			\partial F\left( t,\lambda _{i}\right) }{\partial \lambda _{i}}\\
		&\overset{sign}{=}&
		\frac{%
			\partial F\left( t,\lambda _{i}\right) }{\partial \lambda _{i}}.	
	\end{eqnarray*}%
	For any pair 
	$1\leq i\leq j\leq n$, we have%
		\small{	\begin{eqnarray}\label{eq_fi2}\nonumber
			\left( \lambda _{i}-\lambda _{j}\right) \left( \frac{\partial \varphi
				\left( {\bm\lambda }\right) }{\partial \lambda _{i}}-\frac{\partial
				\varphi \left( {\bm\lambda }\right) }{\partial \lambda _{j}}\right)
			&=&\left( \lambda _{i}-\lambda _{j}\right) 
			\psi _{1}^{\prime }\left(
			\sum_{i=1}^{n}\phi _{1}\left( 1-F\left( t,\lambda _{i}\right) \right)
			\right) \\\nonumber
			&&\left( \frac{1}{\psi _{1}^{\prime }\left( \phi _{1}\left( 1-F\left(
				t,\lambda _{i}\right) \right) \right) }\frac{\partial F\left( t,\lambda
				_{i}\right) }{\partial \lambda _{i}}-\frac{1}{\psi _{1}^{\prime }\left( \phi
				_{1}\left( 1-F\left( t,\lambda _{j}\right) \right) \right) }\frac{\partial
				F\left( t,\lambda _{j}\right) }{\partial \lambda _{j}}\right) \\
			&\overset{sign}{=}&
			-\left( \lambda _{i}-\lambda _{j}\right) \left( \frac{%
				\partial _{\lambda _{i}}F\left( t,\lambda _{i}\right) }{\psi _{1}^{\prime
				}\left( \phi _{1}\left( 1-F\left( t,\lambda _{i}\right) \right) \right) }-%
			\frac{\partial _{\lambda _{j}}F\left( t,\lambda _{j}\right) }{\psi
				_{1}^{\prime }\left( \phi _{1}\left( 1-F\left( t,\lambda _{j}\right) \right)
				\right) }\right)
			\end{eqnarray}}
	since, $\psi _{1}^{\prime }\leq 0$.
	%
	Now, let us denote
	\begin{eqnarray}\label{eq_kappa}
	\kappa(\lambda)=\frac{\partial_{\lambda}F\left( t,\lambda\right)}{1-F\left( t,\lambda\right)}
	\cdot
	\frac{\psi
		_{1}\left( \phi _{1}\left(1- F\left( t,\lambda\right) \right)
		\right)}{\psi
		_{1}^{\prime }\left( \phi _{1}\left( 1-F\left( t,\lambda\right) \right)
		\right)}
	=\kappa_1(\lambda)\cdot \kappa_2(\lambda),
	\end{eqnarray}
	where
	\[\kappa_1(\lambda)=-\frac{\partial\log  \bar{F}\left( t,\lambda\right)}{\partial \lambda} 
	\quad\textup{and}\quad
	\kappa_2(\lambda)=\frac{\psi_{1}}{\psi
		_{1}^{\prime }}
	\Big( \phi _{1}\left(1- F\left( t,\lambda\right) \right)
	\Big).
	\]
	Then, \eqref{eq_fi2} can be rewritten as
	\[\left( \lambda _{i}-\lambda _{j}\right) \left( \frac{\partial \varphi
		\left( {\bm\lambda }\right) }{\partial \lambda _{i}}-\frac{\partial
		\varphi \left( {\bm\lambda }\right) }{\partial \lambda _{j}}\right)
	\overset{sign}{=}
	-\left( \lambda _{i}-\lambda _{j}\right)
	\Big(\kappa(\lambda_{i})-\kappa(\lambda_{j})	\Big),	\]
	since $\psi
	_{1}\left( \phi _{1}\left(1- F\left( t,\lambda\right) \right)
	\right)= \bar{F}\left( t,\lambda\right)$.
	Therefore, we need to prove that $\kappa$ is decreasing.
	
	Note that the function $\kappa_2$ is negative-valued since $\psi
	_{1}^{\prime }\leq 0$.	
	In addition, observe that if $\psi _{1}/\psi _{1}^{\prime }$
	is decreasing  in $t$, which is equivalent to $\psi _{1}$ log-convex, 
	and $F$ is decreasing (increasing), then $\kappa_2$ 
	is increasing (decreasing).
	On the other hand, it is evident that $\kappa_1$ is increasing since $\bar F$ is log-concave in $\lambda$ and it is negative-valued (positive-valued) if 
	$F$ is decreasing (increasing) in $\lambda$.
	Therefore,
	from \eqref{eq_kappa}, we have that $\kappa$
	is decreasing when $F$ is decreasing (increasing) in $\lambda$,  $\bar F$ log-concave in $\lambda$, and $\psi _{1}$ log-convex.
	Thus, from Theorem 3.A.4 in \cite{Marshall:2011}, we know that
	the function $\varphi \left( {\bm\lambda }\right) $ is Schur-convex
	with respect to ${\bm\lambda }$. 
	In other words, $X_{1:n}\leq
	_{ST}X_{1:n}^{\ast }$ if ${\bm\lambda} \overset{w}{\geq }{\bm\lambda }^{\ast
	}\left( {\bm\lambda \geq }_{w}{\bm\lambda }^{\ast }\right) $.
	The rest
	of the proof is the same than in Proposition \ref{prop_02}, hence omitted.
\end{proof}

\begin{remark}\label{rem_04}
		It is worth noting that,
	from \eqref{eq_kappa}, we have that the function $\kappa$
	is increasing when $F$ is increasing (decreasing) and $\bar F$ log-convex in $\lambda$ and $\psi _{1}$ log-concave.
	Thus,
	the function $\varphi \left( {\bm\lambda }\right) $ is Schur-concave
	with respect to ${\bm\lambda }$. 
	Then, if $\widehat C_{\psi _{1}}\geq \widehat C_{\psi _{2}}$ and $F\leq _{ST}G$,	
	the condition
	${\bm\lambda} \overset{w}{\geq }{\bm\lambda }^{\ast
	}\left( {\bm\lambda \geq }_{w}{\bm\lambda }^{\ast }\right) $
	implies
	$X_{1:n}\geq
	_{ST}Y_{1:n}$.
\end{remark}

{From Lemma \ref{lema01}, we know that the MPHR model for $\alpha>1$,
	the SPRH model for $\lambda>0$ when $\tilde{r}$ increasing,
	the SPH model for $\lambda>0$ when ${r}$ increasing and
	the Pareto IV distribution for $\alpha>0$
	verify the conditions on $F$ in Proposition \ref{prop_04}. 
	Whereas, the MPHR model for $0<\alpha\leq 1$,
		the SPH model for $\lambda>0$ when ${r}$ decreasing
	and
	the Pareto IV distribution for $\lambda>0$ when $0<\alpha\leq 1$
		verify the conditions on $F$ in Remark \ref{rem_04}. 
	Therefore, we can obtain comparisons between series systems following
	these type of distributions.}

Now, we assume that the components have extended exponential distributions which belong to the
MPHR model 
as
defined in \eqref{MPHR}. 
From Lemma \ref{lema01}(iii), we know that $\bar F$ is log-convex (log-concave)
in $\lambda$ if $0<\alpha\leq 1$  ($\alpha >1$).
Then, from Proposition \ref{prop_04}, we know that, if $\alpha >1$, then
${\bm\lambda} \overset{w}{\geq }{\bm\lambda }^{\ast
}\left( {\bm\lambda \geq }_{w}{\bm\lambda }^{\ast }\right) $ implies $X_{1:n}\leq
_{ST}Y_{1:n}$ when $F\leq _{ST}G$ and $\psi _{1}$ log-convex. 
And, from the above Remark \ref{rem_04},
if $0<\alpha\leq 1$, then
${\bm\lambda} \overset{w}{\geq }{\bm\lambda }^{\ast
}\left( {\bm\lambda \geq }_{w}{\bm\lambda }^{\ast }\right)$ implies $X_{1:n}\geq
_{ST}Y_{1:n}$ when $F\geq _{ST}G$ and $\psi _{1}$ log-concave.
It is worth mentioning that  Theorem 2 in \cite{Barmalzan-2020} is
therefore not correct.

\bigskip


\begin{example}\label{ex_log}
	Let us consider two series systems $X_{1:2}$ and $X^{*}_{1:2}$ with $F\overset{ST}{=}G$, 	for simplicity. 
	We assume that $F$ is defined as in \eqref{MPHR}
	where $H$ is an exponential distribution with hazard rate $1$
	and 
	$\widehat{C}_{\psi _{i}}$ is a Gumbel copula defined as follows
	\[\widehat{C}_{\psi _{i}}(u,v)=u v 
	\exp \left(-\theta_i \log u \log v    \right),
	\quad \textup{for }
	u,v\in[0,1],
	\textup{ and }	\theta_i\in (0,1],	\]
	with generator
	$\psi _{i}(t)=\exp(\frac{1-e^{t}}{\theta_i})$ 
	for $i=1,2$ (see (4.2.9)  in Table 4.1 in \cite{N06}).
	Firstly, observe that $\log \psi _{1}(t)=(1-e^{t})/\theta_1 $ is concave
	since its second derivative with respect to $t$ is
	$-\theta_{1}^{-1}e^{t}<0$, and therefore $\psi _{1}$ is log-concave.
	Secondly, it is easy to check that 
	$\widehat{C}_{\psi _{1}}\leq (\geq) \widehat{C}_{\psi _{2}}$
	if $\theta_1> (<)\theta_2$, i.e., the Gumbel copula is negative ordered.
	Thirdly, as we show before, from Lemma \ref{lema01}(iii), we know that $\bar F$ is log-convex (log-concave)
	in $\lambda$ if $0<\alpha\leq 1$  ($\alpha >1$), and
	$F$ is increasing in $\lambda$ for any $\alpha>0$.
	Therefore, if $\theta_1<\theta_2$ and $0<\alpha\leq 1$, all the conditions in Remark \ref{rem_04} are satisfied. Set $\theta_1=0.1$ and $\theta_2=0.9$.
		Finally, we take 
	${\bm\lambda}=(1.1, 3)$ and ${\bm\lambda}^{\ast }=(1.9, 2.4)$.
	It is then easy to obtain that
	${\bm\lambda}{\overset{w}{\geq }}{\bm\lambda }^{\ast }$.
	Figure \ref{Fig_log} shows the distribution functions of 
	$X_{1:2}$ (solid line) and $X_{1:2}^{*}$ (dashed line)
	for $\alpha = 0.5$ (left) and 
	for $\alpha = 20$ (right).	
	As it is seen in Figure \ref{Fig_log}(left), 
	$X_{1:2}\geq_{ST}X_{1:2}^{*}$, which coincides with the result of Remark \ref{rem_04}, since all the conditions  
	in that remark are satisfied.
	However, in Figure \ref{Fig_log}(right), the distribution functions cross each other
	and this means that $X_{1:2}$ and $X^{*}_{1:2}$ are not ordered
	since $\bar F$ is not log-convex.
\end{example}

\begin{figure}[H]
	\centering
	\includegraphics[width=6 cm]{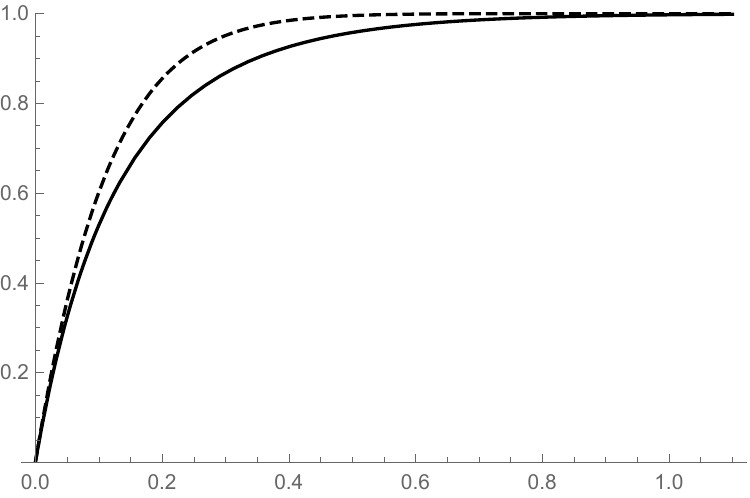}
	\includegraphics[width=6 cm]{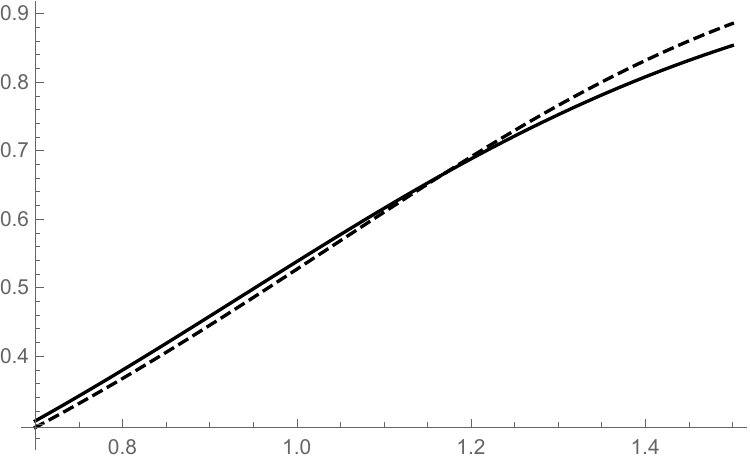}
	\caption{Plots of the distributions functions for  $X_{1:2}$ (solid line) and $X_{1:2}^{*}$ (dashed line)
		for $\alpha = 0.5$ (left) and 
	for $\alpha = 20$ (right) considered in Example \protect\ref{ex_log}.}
	\label{Fig_log}
\end{figure}


\section{Independent case for extreme OS}\label{sec_ind}

In this section, we study comparisons between two parallel or series systems formed by
heterogeneous independent components. 
When the samples are independent, it is well-known that the generator associated to an Archimedean copula is
$\psi(x)=e^{-x}$ for $x\geq 0$ and therefore $\phi(u)=-\log u$ for $u\in[0,1]$.

\begin{proposition}\label{prop_ind1} 
	Let $X_{1},\ldots ,X_{n}$ and $Y_{1},\ldots ,Y_{n}$ be two sets
	of independent nonnegative random variables with $X_{i}\sim F\left( t,\lambda
	_{i}\right) $ and $Y_{i}\sim G\left( t,\lambda _{i}^{\ast }\right) $
	for $i=1,\ldots ,n$.  Then, ${\bm\lambda} \overset{w}%
	{\geq }{\bm\lambda }^{\ast }\left( {\bm\lambda \geq }_{w}{\bm\lambda }%
	^{\ast }\right) $ implies
	\begin{itemize}
		\item[i)] $X_{n:n}\geq _{ST}Y_{n:n}$,
		if $F\left( t,\lambda\right) $
		is log-concave and increasing (decreasing) in $\lambda$,
		and $F\geq_{ST} G$.
		
		\item[ii)] $X_{1:n}\leq _{ST}Y_{1:n}$ 
		and $X_{n:n}\leq _{ST}Y_{n:n}$, 
		if $F\left( t,\lambda\right) $
		is log-convex and decreasing (increasing) in $\lambda$,
		and $F\leq_{ST} G$.
	\end{itemize}
\end{proposition}

\begin{proof}
	\begin{itemize}
		\item[i)] To prove that $X_{n:n}\geq _{ST}Y_{n:n}$, observe that, 
		for independent samples,
		the generator associated to an Archimedean copula is
		$\psi(x)=e^{-x}$ and therefore $\phi(x)=-\log x$.
		Then, $\phi(F(t,\lambda))=-\log F(t,\lambda)$ is convex 
		with respect to $\lambda$ since
		$F(t,\lambda)$ is log-concave, and therefore,
		from Proposition \ref{prop_01}, we have the desired result.
		\item[ii)] In order to prove that $X_{1:n}\leq _{ST}Y_{1:n}$,
		we know,
		from Proposition 21.A.6 in \cite{Marshall:2007}, that if $%
		F\left( t,\lambda \right) $ is log-convex, then $F\left(
		t,\lambda \right) $ is convex. 
		Then, the result is obtained from Theorem \ref{th02} where $D_1 = D_2$. 
		
		On the other hand, $\phi(F(t,\lambda))=-\log F(t,\lambda)$ is concave 
		with respect to $\lambda$ since
		$F(t,\lambda)$ is log-convex, and therefore,
		from Remark \ref{rem_01}, we have $X_{n:n}\leq _{ST}Y_{n:n}$.
	\end{itemize}
\end{proof}

\bigskip

The following two examples demonstrate the theoretical result of
Proposition \ref{prop_ind1}.

\begin{example}\label{ex_indp1}
	Let $X_1,\ldots,X_4$ and $Y_1,\ldots,Y_4$ be two sets of independent nonegative random variables with  $X_{i}\sim F\left( t,\lambda
	_{i}\right) $ and $Y_{i}\sim G\left( t,\lambda _{i}^{\ast }\right) $
	for $i=1,\ldots ,4$. 
	{Let us consider that $F$ belongs to the MPRHR model as defined in \eqref{MPRHR}
			where $H_F$ is an exponential distribution with hazard rate $1$, whereas 
			$G$ belongs to the proportional reversed hazard rate (PRHR) model
			where $H_G$ is an exponential distribution with hazard rate $2$.}
			From Lemma \ref{lema01}(ii), we know that $F$ is decreasing in $\lambda$
		for any $\alpha>0$ and it is log-concave if $\alpha>1$.
		{It is easy to check that $F\geq_{ST}G$ when $\alpha=2$ in the MPRHR model, however, $F$ and $G$ cross each other when 
			$\alpha=15$ in the MPRHR model.}
	We take $\bm\lambda=(1.7, 2, 3.2, 4.5)$ and
	$\bm\lambda^{*}=(1.3, 2.7, 2.8, 3)$. It is evident that 
	 $\bm\lambda\geq_{w}\bm\lambda^{*}$.
	 Therefore, {for $\alpha=2$}, all the conditions in Proposition \ref{prop_ind1}(i) are satisfied.
	 Figure \ref{Fig_indp1}(left) shows the distribution functions of 
	$X_{4:4}$ (solid line) and $Y_{4:4}$ (dashed line).
	As it can be seen, this plot is according to Proposition \ref{prop_ind1}(i), that is,
	 $X_{4:4}\geq _{ST}Y_{4:4}$.
	 However, for $\alpha=15$, $F$ and $G$ cross each other and from Figure \ref{Fig_indp1}(right),
		it can be seen that $X_{4:4}$ and $Y_{4:4}$ are not ordered, i.e., $F\geq_{ST}G$ is a necessary condition.
\end{example}

\begin{figure}[H]
	\centering
	\includegraphics[width=5.5 cm]{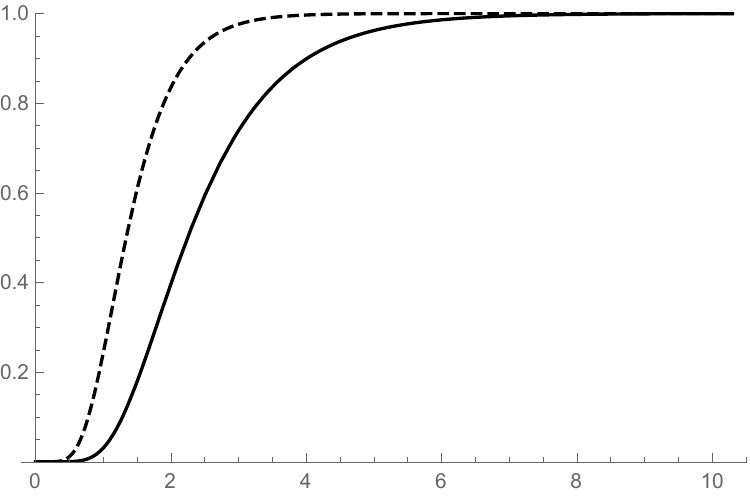}
	\includegraphics[width=5.5 cm]{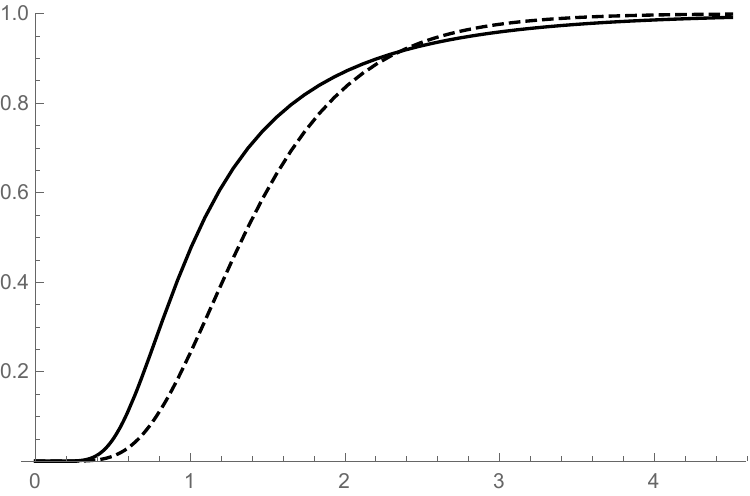}
	\caption{Plots of the distributions functions for  
		parallel (left) systems when $F\geq_{ST}G$ and 	parallel  systems when $F\ngeq_{ST}G$ (right),
		considered in Example \protect\ref{ex_indp1}.}
	\label{Fig_indp1}
\end{figure}

\begin{example}\label{ex_indp2}
	Now, we consider the same random variables than in Example \ref{ex_indp1}
	but, in this case, we assume that the baseline distribution function $H$ is an exponential distribution with hazard rate $6$ for $F$, and 
	with hazard rate $3$ for $G$. 
	In addition, we consider $0<\alpha\leq 1$, then $F$ is log-convex from Lemma \ref{lema01}(ii).
		{It is easy to check that $F\leq_{ST}G$ when $\alpha=0.5$ in the MPRHR model, however, $F$ and $G$ cross each other when 
		$\alpha=0.02$ in the MPRHR model.}
	Set $\bm\lambda=(0.5, 1.2, 2.2, 3)$ and
	$\bm\lambda^{*}=(0.9, 1.5, 1.8, 3.2)$. It is evident that 
	$\bm\lambda\overset{w}{\geq }\bm\lambda^{*}$.
	Therefore, {for $\alpha=0.5$}, all the conditions in Proposition \ref{prop_ind1}(ii) are satisfied.
	Figure \ref{Fig_indp2}(left) shows the distribution functions of 
	$X_{1:4}$ (solid line) and $Y_{1:4}$ (dashed line),
	and Figure \ref{Fig_indp2}(center) displays the distribution functions of 
	$X_{4:4}$ (solid line) and $Y_{4:4}$ (dashed line).
	As it can be seen, both plots are according to Proposition \ref{prop_ind1}(ii), that is,
	$X_{1:4}\leq _{ST}Y_{1:4}$ 
	and $X_{4:4}\leq _{ST}Y_{4:4}$. 
		{ However, for $\alpha=0.02$, $F$ and $G$ cross each other and from Figure \ref{Fig_indp2}(right),
		it can be seen that $X_{4:4}$ and $Y_{4:4}$ are not ordered, i.e., $F\leq_{ST}G$ is a necessary condition.}
\end{example}

\begin{figure}[H]
\centering
\includegraphics[width=5.5 cm]{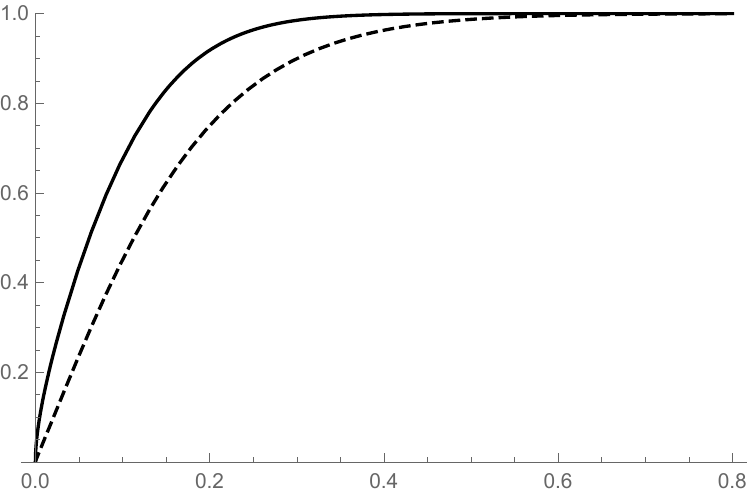}
\includegraphics[width=5.5 cm]{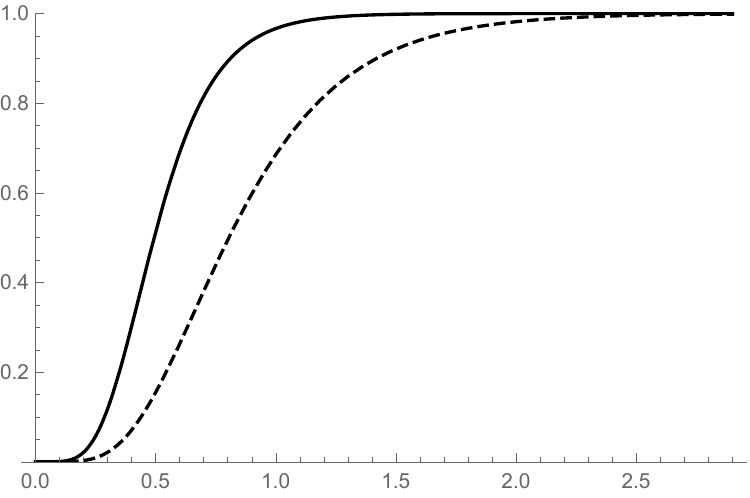}
\includegraphics[width=5.5 cm]{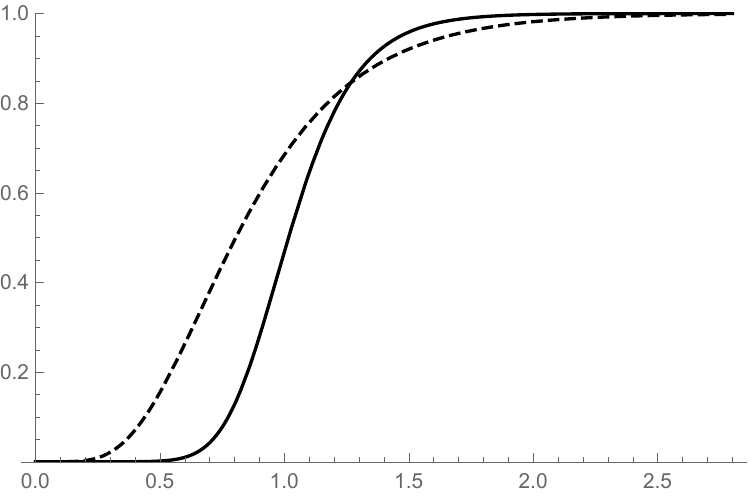}
\caption{Plots of the distributions functions for  
	series (left) and parallel (center) systems when $F\leq_{ST}G$ and 	parallel  systems when $F\nleq_{ST}G$ (right),
	considered in Example \protect\ref{ex_indp2}.}
\label{Fig_indp2}
\end{figure}

In the following result we compare series and parallel systems when the reliability function is log-convex or log-concave with respect to a parameter.

\begin{proposition}\label{prop_ind2}
	Let $X_{1},\ldots ,X_{n}$ and $Y_{1},\ldots ,Y_{n}$ be two sets
	of independent nonnegative random variables with $X_{i}\sim F\left( t,\lambda
	_{i}\right) $ and $Y_{i}\sim G\left( t,\lambda _{i}^{\ast }\right) $
	for $i=1,\ldots ,n$.  Then, ${\bm\lambda} \overset{w}%
	{\geq }{\bm\lambda }^{\ast }\left( {\bm\lambda \geq }_{w}{\bm\lambda }%
	^{\ast }\right) $ implies
	\begin{itemize}
		\item[i)] $X_{1:n}\leq _{ST}Y_{1:n}$,
		if $\bar F\left( t,\lambda\right) $
		is log-concave and increasing (decreasing) in $\lambda$,
		and $F\leq_{ST} G$.
		
		\item[ii)] $X_{1:n}\geq _{ST}Y_{1:n}$ 
		and $X_{n:n}\geq _{ST}Y_{n:n}$, 
		if $\bar F\left( t,\lambda\right) $
		is log-convex and decreasing (increasing) in $\lambda$,
		and $F\geq_{ST} G$.
	\end{itemize}
\end{proposition}

\begin{proof}
	\begin{itemize}
		\item[i)] To prove that $X_{1:n}\leq _{ST}Y_{1:n}$, observe that, 
		for independent samples, as before,
		$\psi(x)=e^{-x}$ and $\phi(x)=-\log x$.
		Then, $\phi(1-F(t,\lambda))=-\log \bar F(t,\lambda)$ is convex 
		with respect to $\lambda$ since
		$\bar F(t,\lambda)$ is log-concave, and therefore,
		from Proposition \ref{prop_02}, we have the desired result.		
		\item[ii)]			
		In order to prove that $X_{1:n}\geq _{ST}Y_{1:n}$,
		we know,  $\phi(1-F(t,\lambda))=-\log \bar F(t,\lambda)$ is concave 
		with respect to $\lambda$ since
		$\bar F(t,\lambda)$ is log-convex, and therefore,
		from Remark \ref{rem_02}, we have the desired result.
		
		On the other hand,			
		from Proposition 21.A.6 in \cite{Marshall:2007}, we know that if $%
		\bar F\left( t,\lambda \right) $ is log-convex, then $\bar F\left(
		t,\lambda \right) $ is convex, and so $F\left(
		t,\lambda \right) $ concave. 
		Then, the result is obtained from Theorem \ref{th01} where $D_1 = D_2$. 		
	\end{itemize}
\end{proof}



\section{Concluding remarks}\label{sec_conclusions}

We have presented results to compare the reliability of coherent systems 
with dependent components where its random lifetimes follows any distribution 
function dependent on a parameter. We have investigated 
three cases: general coherent systems with heterogeneous and dependent components where the dependency structure can be defined by any copula,
parallel and series systems with heterogenous and dependent components
assembled by the family of the Archimedean copulas, and finally, parallel and series systems for independent but not identically distributed components.
In addition, we have showed how to apply these results to particular families of distributions 
widely used in the literature (see, for instance, \cite{TL11, LK13}).

In future studies, we hope to report our findings regarding comparisons between the hazard (reversed) rate functions among coherent systems with heterogeneous, dependent  and distribution-free components. 
To the best of our knowledge, there exits some recent works about stochastic comparisons in the sense of the hazard (reversed) rate orders between parallel and series systems with associated Archimedean copulas and for some type of distribution functions. However, the study of these types of stochastic orders, under some majorization conditions, between coherent systems for any copula and any distribution function depending on a parameter is still an open problem.

\section*{Acknowledgements}
The author thanks the three anonymous
reviewers  and Editors for helpful comments on earlier versions of this manuscript.

\section*{Funding}

This work was partially supported by the Spanish Ministerio de Econom\'{i}a y Competitividad under grant ECO2015-70331-C2-1-R and Ministerio de Ciencia e Innovaci\'{o}n of Spain under grant 
PID2019-108079GB-C22/AEI/10.13039/501100011033.

\bigskip



\section*{Appendix}

First we recall some concepts of majorization. Let $\{x_{1:n},x_{2:n},\ldots
,x_{n:n}\}$ denote the increasing arrangement of the components of the
vector ${\bm x}=\left( x_{1},x_{2},\ldots ,x_{n}\right)$.

\begin{definition}\label{def_maj}
	The vector ${\bm x}$ is said to be
	
	\begin{itemize}
		\item[a)] majorized by the vector ${\bm y}$, denoted by ${\bm x}\overset{%
			\text{m}}{\leq }{\bm y}$, if
		\begin{equation*}
		\sum_{i=1}^{j}x_{i:n}\geq \sum_{i=1}^{j}y_{i:n},\quad \text{for }j=1,\ldots
		,n-1\quad \text{ and }\quad \sum_{i=1}^{n}x_{i:n}=\sum_{i=1}^{n}y_{i:n},
		\end{equation*}
		
		\item[b)] weakly supermajorized by the vector ${\bm y}$, denoted by ${\bm x%
			\overset{\text{w}}{\leq }\bm y}$, if
		\begin{equation*}
		\sum_{i=1}^{j}x_{i:n}\geq \sum_{i=1}^{j}y_{i:n},\quad \text{for }j=1,\ldots
		,n,
		\end{equation*}
		
		\item[c)] weakly submajorized by the vector ${\bm y}$, denoted by ${\bm x }
		\leq_{\text{w}} {\bm y}$, if
		\begin{equation*}
		\sum_{i=j}^{n}x_{i:n}\leq \sum_{i=j}^{n}y_{i:n},\quad \text{for }j=1,\ldots
		,n.
		\end{equation*}		
	\end{itemize}
\end{definition}

It is known that the majorization order implies both the weak submajorization and supermajorization orders.
The converse is, however, not true.

\bigskip

To finish this section, we study the monotony, convexity (concavity) and log-convexity (log-concavity) for different types of well known families of distributions. These results are given in the following lemma. {Some references about parametric families of distributions are, for instance, \cite{Marshall:2007, Balakrishnan-2018} and \cite{Silva-2019}.}

\begin{lemma}\label{lema01}
	Let $H$ a baseline distribution function defined in $t>0$ with associated hazard rate $r$ and reversed hazard rate $\tilde{r}$. Then,
	\begin{itemize}
		\item[i)] The type II half logistic scale (TIIHL-RS) distribution
		\begin{equation}\label{eq_tiil}
		F(t,\lambda ,\alpha )=\frac{2H^{\alpha }\left( \lambda t\right) }{%
			1+H^{\alpha }\left( \lambda t\right) },\quad\text{ for }\alpha ,\lambda >0,
		\end{equation}%
		is decreasing, convex and log-concave in $\alpha $,
		and it is increasing in $\lambda $, convex  if $\tilde{r}$ is increasing in $t$, and
		log-concave  if $\tilde{r}$ is decreasing in $t$.	
		
		\item[ii)] The reversed modified proportional hazard (MPRHR) distribution
		\begin{equation}\label{MPRHR}
		F\left( t,\lambda \right) =\frac{\alpha H^{\lambda }(t)}{1-\bar{\alpha }%
			H^{\lambda }(t)},\quad\text{for }\alpha ,\lambda >0
		\end{equation}%
		where $\bar{\alpha }=1-\alpha$, 
		is decreasing in $\lambda$ for $\alpha>0$, and
		it is log-convex (log-concave)  in $%
		\lambda $
		if $0<\alpha \leq 1$  ($\alpha >1$).
		
		\item[iii)] The modified proportional hazard (MPHR) distribution
		\begin{equation}\label{MPHR}
		F\left( t,\lambda \right) =\frac{1-\bar{H}^{\lambda }(t)}{1-\bar{%
				\alpha }\bar{H}^{\lambda }(t)},\quad\text{ for }\alpha ,\lambda>0,
		\end{equation}%
		where $\bar{\alpha }=1-\alpha$, 
		is increasing and concave in $\lambda$ for $\alpha >0$.
		The associated reliability function, $\bar{F}$, is log-convex (log-concave)  in $%
		\lambda $
		if $0<\alpha \leq 1$  ($\alpha >1$).
		\item[iv)] The scale proportional reversed hazard (SPRH) or resilience-scale (RS) distribution
		\begin{equation}\label{SPRH}
		F(t,\lambda ,\alpha )=H^{\alpha }\left( \lambda t\right),
		\quad\text{ for }\alpha ,\lambda>0,
		\end{equation}%
		is decreasing and convex in $\alpha$, and 
		it is increasing in $\lambda$ and log-convex (log-concave) if
		$\tilde{r}$ increasing (decreasing)
		in $t$. In addition, $\bar F$ is log-concave in $\lambda$ if $\tilde r$ increasing.
		\item[v)] The scale proportional hazard (SPH) distribution
		\begin{equation}\label{SPH}
		F(t,\lambda,\alpha)=1-\bar{H}^{\alpha}(\lambda t),%
		\quad\text{ for }\alpha ,\lambda>0,
		\end{equation}%
		it is increasing and concave in $\alpha$ and it is increasing in $\lambda$ and concave if ${r}$ is decreasing in $t>0$. 
		In addition, $\bar F$ is log-convex (log-concave) in $\lambda$ if $r$ decreasing (increasing).
		\item[vi)] The Pareto IV (P-IV) disribution 
		\begin{equation}\label{pareto}
		F(t,\lambda ,\alpha ,\beta )=1-\left( 1+\left( \lambda t\right) ^{\alpha
		}\right) ^{-\beta },\text{ for }\lambda ,\alpha ,\beta>0,
		\end{equation}%
		is increasing in $\lambda$ and in $\beta $
		and it is not monotone in $\alpha$. 
		The reliability function is log-convex in $\lambda$ if $0<\alpha \leq 1$
		and $\beta>0$,
		log-concave in $%
		\alpha $ for any $\lambda, \beta >0$.
		\item[vii)] The exponentiated Kumaraswamy-G (EK-G) distribution
		\begin{equation}\label{EKG}
		F(t,\alpha ,\beta ,\gamma )=\left( 1-\left( 1-H^{\alpha }\left( t\right)
		\right) ^{\beta }\right) ^{\gamma }\text{ for }\alpha ,\beta ,\gamma>0
		\end{equation}%
		is decreasing in $\alpha$ for all $\beta ,\gamma>0$ and  convex in $\alpha$ when $0<\beta,\gamma \leq 1$; it is increasing in $\beta$ 
		for all $\alpha,\gamma>0$ and concave when $%
		\gamma \leq 1$; it  is decreasing and convex in $\gamma$
		for all $\alpha,\beta>0$.
		\item[viii)] The exponentiated location-scale (EL-S) distribution
		\begin{equation}\label{ELS}
		F(t,\lambda ,\alpha ,\theta )=H^{\alpha }\left( \frac{t-\lambda }{\theta }%
		\right),\quad \text{ for }t>\lambda ,
		\end{equation}%
		where $\lambda \in\mathbb{R} $ and $\alpha ,\theta >0$, is
		decreasing in $\alpha$, decreasing in $\lambda $ and log-concave in $\lambda $ if $\tilde{r}$ decreasing in $t$, it
		is decreasing in $\theta $ and log-convex (log-concave) in $\theta $ if $x\tilde{r}\left( x\right) $ is increasing (decreasing) in $x$.
	\end{itemize}
\end{lemma}

\begin{proof} 
	\begin{itemize}
		\item[i)] The partial derivative of \eqref{eq_tiil} with respect to $\alpha $ is%
		\begin{eqnarray*}
			\frac{\partial F(t,\lambda ,\alpha )}{\partial \alpha } &=&\frac{2}{\left(
				1+H^{\alpha }\left( \lambda t\right) \right) ^{2}}\left( H^{\alpha }\left(
			\lambda t\right) \log H\left( \lambda t\right) \left( 1+H^{\alpha }\left(
			\lambda t\right) \right) -H^{\alpha }\left( \lambda t\right) H^{\alpha
			}\left( \lambda t\right) \log H\left( \lambda t\right) \right)  \\
			&=&\frac{2}{\left( 1+H^{\alpha }\left( \lambda t\right) \right) ^{2}}\log
			H\left( \lambda t\right) H^{\alpha }\left( \lambda t\right) \leq 0,
		\end{eqnarray*}%
		and the second partial derivative is%
		\begin{eqnarray*}
			\frac{\partial ^{2}F(t,\lambda ,\alpha )}{\partial \alpha ^{2}} &\overset{sign}{=}&
			H^{\alpha
			}\left( \lambda t\right) \left(\log H\left( \lambda t\right)\right)^{2} \left( 1+H^{\alpha
			}\left( \lambda t\right) \right) ^{2}-
			2\left(\log H\left( \lambda t\right) H^{\alpha}\left( \lambda t\right)\right)^{2} \left( 1+H^{\alpha }\left( \lambda t\right)
			\right) 
			\\
			&\overset{sign}{=}&
			H^{\alpha }\left( \lambda t\right) \left(\log H\left( \lambda
			t\right)\right)^{2} \left( 1+H^{\alpha }\left( \lambda t\right) \right) \Big( 
			1+H^{\alpha }\left( \lambda t\right)  -2H^{\alpha }\left( \lambda
			t\right)  \Big)  \\
			&=& 
			H^{\alpha }\left( \lambda t\right) \left(\log H\left( \lambda
			t\right)\right)^{2} \left( 1+H^{\alpha }\left( \lambda t\right) \right) \Big( 
			1  -H^{\alpha }\left( \lambda
			t\right)  \Big)
			\geq 0.
		\end{eqnarray*}%
		Therefore, $F(t,\lambda ,\alpha )$ is decreasing and convex in $\alpha$.  Next, we study the monotony of $\log F(t,\lambda ,\alpha )$
		with respect to $\alpha $. It is easy to obtain%
		\begin{equation*}
		\frac{\partial \log F(t,\lambda ,\alpha )}{\partial \alpha }=\frac{\log
			H\left( \lambda t\right) }{1+H^{\alpha }\left( \lambda t\right) }
		\end{equation*}%
		and
		\begin{equation*}
		\frac{\partial^{2} \log F(t,\lambda ,\alpha )}{\partial \alpha^{2} }
		\overset{sign}{=}-\left( \log
		H\left( \lambda t\right) \right) ^{2}H^{\alpha }\left( \lambda t\right) \leq
		0,
		\end{equation*}%
		i.e, $F(t,\lambda ,\alpha )$ is log-concave in $\alpha $.
		
		On the other hand, the derivative of \eqref{eq_tiil} with respect to $%
		\lambda $ is%
		\begin{eqnarray*}
			\frac{\partial F(t,\lambda ,\alpha )}{\partial \lambda } &=&\frac{2}{\left(
				1+H^{\alpha }\left( \lambda t\right) \right) ^{2}}\left( \alpha t H^{\alpha
			}\left( \lambda t\right) \tilde{r}\left( \lambda t\right) \left(
			1+H^{\alpha }\left( \lambda t\right) \right) - \alpha t
			\left(H^{\alpha }\left( \lambda t\right)\right)^{2}
			\tilde{r}\left(\lambda t\right) \right)  \\
			&=&\frac{2\alpha
				t}{\left( 1+H^{\alpha }\left( \lambda t\right) \right) ^{2}}H^{\alpha }\left( \lambda t\right) \tilde{r}\left( \lambda t\right)
			\left( 1+H^{\alpha }\left( \lambda t\right) -H^{\alpha }\left( \lambda
			t\right) \right)  \\
			&=&\frac{2\alpha
				t}{\left( 1+H^{\alpha }\left( \lambda t\right) \right) ^{2}}
			H^{\alpha }\left( \lambda t\right) \tilde{r}\left( \lambda t\right)
			\geq 0,
		\end{eqnarray*}%
		where $\tilde{r}$ is the reversed hazard rate of the baseline distribution function $H$, and the second partial derivative is%
			\begin{eqnarray*}
		\frac{\partial ^{2}F(t,\lambda ,\alpha )}{\partial \lambda ^{2}}
		&\overset{sign}{=}&
		\big( \alpha \tilde{r}^{2}\left( \lambda t\right) +%
		\tilde{r}^{\prime }\left( \lambda t\right) \big) \big( 1+H^{\alpha
		}\left( \lambda t\right) \big) -2\alpha H^{\alpha }\left( \lambda t\right) 
		\tilde{r}^{2}\left( \lambda t\right)  \\
		&=&\tilde{r}^{\prime }\left( \lambda t\right) \big( 1+H^{\alpha }\left(
		\lambda t\right) \big) +\alpha \tilde{r}^{2}\left( \lambda t\right)
		-\alpha H^{\alpha }\left( \lambda t\right) \tilde{r}^{2}\left( \lambda
		t\right)  \\
		&=&\tilde{r}^{\prime }\left( \lambda t\right) \big( 1+H^{\alpha }\left(
		\lambda t\right) \big) +\alpha \tilde{r}^{2}\left( \lambda t\right)
		\left( 1-H^{\alpha }\left( \lambda t\right) \right) \geq 0
	\end{eqnarray*}
	whenever $\tilde{r}^{\prime }\left( \lambda t\right) \geq 0$. Thus, $%
		F(t,\lambda ,\alpha )$ is increasing and convex in $\lambda $ if $\tilde{%
			r}$ is increasing in $t$. 
		Note that, as \cite{Marshall:2007} proved
		(see p. 179), nonnegative random variables cannot have distributions with
		increasing reversed hazard rate.
		
		In order to study
		the log-convexity or log-concavity of \eqref{eq_tiil} with respect
		to $\lambda $, we need to prove that 
		\begin{equation*}
		\frac{\partial \log F(t,\lambda ,\alpha )}{%
			\partial \lambda }=\frac{\alpha t\text{ }\tilde{r}\left( \lambda
			t\right) }{1+H^{\alpha }\left( \lambda t\right) }
		\end{equation*}%
		is increasing or decreasing, respectively. Note that%
		\begin{eqnarray*}
			\frac{\partial^{2} \log F(t,\lambda ,\alpha )}{\partial \lambda^{2} }&\overset{sign}{=}&\alpha
			t^{2}\tilde{r}^{\prime }\left( \lambda t\right) \left( 1+H^{\alpha
			}\left( \lambda t\right) \right) -\left( \alpha t\right) ^{2}H^{\alpha
			}\left( \lambda t\right) \tilde{r}^{2}\left( \lambda t\right) \\
			&\overset{sign}{=}&\tilde{r}^{\prime }\left( \lambda t\right) \left(
			1+H^{\alpha }\left( \lambda t\right) \right) -\alpha H^{\alpha }\left(
			\lambda t\right) \tilde{r}^{2}\left( \lambda t\right) 
			\leq 0,
		\end{eqnarray*}%
		if $\tilde r^{\prime }\left( \lambda t\right) \leq 0$, i.e, $F(t,\lambda ,\alpha )$
		is log-concave in $\lambda $ if $\tilde{r}$ is decreasing in $t$.
		\item[ii)] The partial 
		derivative of \eqref{MPRHR} with respect to $\lambda $ is%
		\begin{eqnarray*}
			\frac{\partial F(t,\lambda )}{\partial \lambda } &=&\frac{\alpha }{\left( 1-%
				\bar{\alpha }H^{\lambda }(t)\right) ^{2}}\Big( H^{\lambda }(t)\log
			H(t)\left( 1-\bar{\alpha }H^{\lambda }(t)\right) +\alpha H^{\lambda }(t)%
			\bar{\alpha }H^{\lambda }(t)\log H(t)\Big)  \\
			&=&\frac{\alpha }{\left( 1-\bar{\alpha }H^{\lambda }(t)\right) ^{2}}%
			H^{\lambda }(t)\log H(t)\left( 1-\bar{\alpha }H^{\lambda }(t)+\alpha 
			\bar{\alpha }H^{\lambda }(t)\right)  \\
			&=&\frac{\alpha }{\left( 1-\bar{\alpha }H^{\lambda }(t)\right) ^{2}}%
			H^{\lambda }(t)\log H(t)\left( 1-\bar{\alpha} ^{2}H^{\lambda }(t)\right) \leq 0
		\end{eqnarray*}%
		for $\alpha >0$, i.e, the function $F(t,\lambda )$ is decreasing in $%
		\lambda $. Now, observe that
		\[\log F(t,\lambda )=\log \alpha + \log H^{\lambda}(t)-\log \left(1-\bar \alpha H^{\lambda}(t) \right).		\]
		Then, the first partial derivative of $\log F$ with respect to $\lambda$ is
		\begin{equation*}
		\frac{\partial \log F(t,\lambda )}{\partial \lambda }=%
		\log H(t)+\frac{\bar\alpha H^{\lambda }(t) \log H(t)}{1-\bar{%
				\alpha }H^{\lambda }(t)}	
		=
		\frac{\log H(t) }{1-\bar{%
				\alpha }H^{\lambda }(t)},
		\end{equation*}%
		then%
		\begin{eqnarray*}
			\frac{\partial^{2} \log F(t,\lambda ,\alpha )}{\partial \lambda^{2} }
			&\overset{sign}{=}&
			\bar\alpha H^{\lambda }(t)
			\left(\log H(t)\right)^{2}.					 
		\end{eqnarray*}%
		Therefore, $F(t,\lambda )$ is log-convex in $%
		\lambda $	
		if $0<\alpha \leq 1$ and it is log-concave in $%
		\lambda $ if $\alpha >1$.
		\item[iii)] The
		partial derivative of \eqref{MPHR} with respect to $\lambda $ is%
		\begin{eqnarray*}
			\frac{\partial F(t,\lambda )}{\partial \lambda }
			&=&\frac{\bar{H}^{\lambda }(t)\log \bar{H}(t)}{\left( 1-\bar{\alpha }\bar{H}^{\lambda }(t)\right) ^{2}%
			}\Big( -\left( 1-\bar{%
				\alpha }\bar{H}^{\lambda }(t)\right) +\bar{\alpha }\left( 1-%
			\bar{H}^{\lambda }(t)\right) \Big)  \\
			&=&\frac{\bar{H}^{\lambda }(t)\log \bar{H}(t)}{\left( 1-\bar{\alpha }\bar{H}^{\lambda }(t)\right) ^{2}}		
		\left( \bar{\alpha }%
			-1\right) \\
			&\overset{sign}{=}&
			-\alpha \log \bar{H}(t)
			\geq 0,
		\end{eqnarray*}%
		since $\alpha >0$. Thus, $%
		F(t,\lambda )$ is increasing in $\lambda $. The second partial derivative of 
		$F(t,\lambda )$ with respect to $\lambda $ is%
		\begin{eqnarray*}
			\frac{\partial ^{2}F(t,\lambda )}{\partial \lambda ^{2}}&\overset{sign}{=}&
			\left( \log \bar{H}(t)\right) ^{2}\left( \bar{\alpha }-1\right) 
			\bar{H}^{\lambda }(t)\left( 1-\bar{\alpha }\bar{H}^{\lambda
			}(t)\right) ^{2}\\
		&&\hspace{0.8cm}
		+2\left( \bar{\alpha }-1\right) \bar{\alpha }%
			\left( \bar{H}^{\lambda }(t)\log \bar{H}(t)\right) ^{2}\left( 1-%
			\bar{\alpha }\bar{H}^{\lambda }(t)\right)  \\
			&\overset{sign}{=}&\left( \bar{\alpha }-1\right) \Big( 1-\bar{%
				\alpha }\bar{H}^{\lambda }(t)+2\bar{\alpha }\bar{H}^{\lambda
			}(t)\Big)  \\
			&=&\left( \bar{\alpha }-1\right) \big( 1+\bar{\alpha }\bar{H}%
			^{\lambda }(t)\big) =-\alpha \left( 1+\bar{\alpha }\bar{H}%
			^{\lambda }(t)\right) \leq 0.
		\end{eqnarray*}%
		Then, $F(t,\lambda )$ is concave in $\lambda $.
		On the other hand, it is easy to get
		\[\log \bar F(t,\lambda )=\log\alpha+\log \bar{H}^{\lambda }(t)
		-\log \left( 1-\bar{\alpha }\bar{H}%
		^{\lambda }(t)\right).
		\]
		Therefore
		\begin{equation*}
		\frac{\partial \log \bar F(t,\lambda )}{\partial \lambda } =	\log \bar H(t)+\frac{\bar\alpha \bar H^{\lambda }(t) \log \bar H(t)}{1-\bar{%
				\alpha }\bar H^{\lambda }(t)}	
		=
		\frac{\log \bar H(t) }{1-\bar{%
				\alpha } \bar H^{\lambda }(t)},
		\end{equation*}%
		then%
		\begin{eqnarray*}
			\frac{\partial^{2} \log \bar F(t,\lambda ,\alpha )}{\partial \lambda^{2} }
			&\overset{sign}{=}&
			\bar\alpha \bar H^{\lambda }(t)
			\left(\log \bar H(t)\right)^{2}.					 
		\end{eqnarray*}%
		Therefore, $\bar F(t,\lambda )$ is log-convex in $%
		\lambda $	
		if $0<\alpha \leq 1$ and it is log-concave in $%
		\lambda $ if $\alpha >1$.
		\item[iv)] Let us consider $F(t,\lambda ,\alpha )$ defined as in \eqref{SPRH}.
		It is easy to check that%
		\begin{equation*}
		\frac{\partial F(t,\lambda ,\alpha )}{\partial \alpha }=H^{\alpha }\left(
		\lambda t\right) \log H\left( \lambda t\right) \leq 0
		\quad\text{ and }\quad
		\frac{%
			\partial ^{2}F(t,\lambda ,\alpha )}{\partial \alpha ^{2}}=H^{\alpha }\left(
		\lambda t\right) \left( \log H\left( \lambda t\right) \right) ^{2}\geq 0.
		\end{equation*}%
		Therefore, $F(t,\lambda ,\alpha )$ is decreasing and convex in $\alpha $. 
		
		On the other hand,%
		\begin{equation*}
		\frac{\partial F(t,\lambda ,\alpha )}{\partial \lambda }=\alpha t H^{\alpha
		}\left( \lambda t\right) \tilde r\left( \lambda t\right) \geq 0,
		\end{equation*}%
		i.e., $F(t,\lambda ,\alpha )$ is increasing in $\lambda $, and 
		\begin{equation*}
		\frac{\partial^{2} F(t,\lambda ,\alpha )}{\partial \lambda^{2} }=
		\alpha t\Big(  \alpha t H^{\alpha}(t) \tilde r^{2}( \lambda t) +
		t H^{\alpha}(t) \tilde r'(\lambda t)
		\Big)\overset{sign}{=}
		\alpha \tilde r^{2}( \lambda t) + \tilde r'(\lambda t).
		\end{equation*}%
		Then, $F$ is convex in $\lambda$ if $\tilde r$ is increasing in $t$.
		Now, we study the
		monotony of $\log F(t,\lambda ,\alpha )$ with respect to $%
		\lambda $. It is easy to obtain
		$\log F(t,\lambda ,\alpha )=\alpha \log H(\lambda t)$, then
		\begin{equation*}
		\frac{\partial \log F(t,\lambda ,\alpha )}{\partial \lambda }=\alpha t%
		\tilde{r}\left( \lambda t\right) .
		\end{equation*}%
		Hence, $\tilde{r}\left( \lambda t\right) $ is increasing (decreasing) in 
		$\lambda ,$ which is equivalent to $\tilde{r}$ increasing (decreasing)
		in $t$, implies $F(t,\lambda ,\alpha )$ is log-convex (log-concave) in\ $%
		\lambda $. 
		Next, observe that 
		\[\log \bar F(t,\lambda ,\alpha )=\log \left(1-H^{\alpha
		}\left( \lambda t\right)\right).\]
		Then 
		\begin{equation*}
		\frac{\partial \log \bar F(t,\lambda ,\alpha )}{\partial \lambda }=
		-\alpha t%
		\tilde{r}\left( \lambda t\right)\frac{H^{\alpha
			}\left( \lambda t\right)}{1-H^{\alpha
			}\left( \lambda t\right)}.
		\end{equation*}%
		Thus, the above ratio is decreasing if $\tilde r$ is increasing since 
		$H^{\alpha
		}\left( \lambda t\right)/(1-H^{\alpha 
		}\left( \lambda t\right))$ is positive-valued and increasing.
		Therefore, $\bar F$ is log-concave in $\lambda$ if $\tilde r$ increasing.
		\item[v)] Let us consider $F(t,\lambda ,\alpha )$ defined as in \eqref{SPH}
		It is easy to check that%
		\begin{equation*}
		\frac{\partial F(t,\lambda ,\alpha )}{\partial \alpha }=- \bar H^{\alpha }\left(
		\lambda t\right) \log \bar H\left( \lambda t\right) \geq 0
		\quad\text{ and }\quad
		\frac{%
			\partial ^{2}F(t,\lambda ,\alpha )}{\partial \alpha ^{2}}=- \bar H^{\alpha }\left(
		\lambda t\right) \left( \log \bar H\left( \lambda t\right) \right) ^{2}\leq 0.
		\end{equation*}%
		Therefore, $F(t,\lambda ,\alpha )$ is increasing and concave in $\alpha $. 
		
		On the other hand,%
		\begin{equation*}
		\frac{\partial F(t,\lambda ,\alpha )}{\partial \lambda }=\alpha t \bar H^{\alpha
		}\left( \lambda t\right)  r\left( \lambda t\right) \geq 0,
		\end{equation*}%
		i.e., $F(t,\lambda ,\alpha )$ is increasing in $\lambda $, and 
		\begin{equation*}
		\frac{\partial^{2} F(t,\lambda ,\alpha )}{\partial \lambda^{2} }=
		\alpha t\Big( - \alpha t \bar H^{\alpha}(t) {r}^{2}( \lambda t) +
		t \bar H^{\alpha}(t) {r}'(\lambda t)
		\Big)\overset{sign}{=}
		-\alpha {r}^{2}( \lambda t) + {r}'(\lambda t).
		\end{equation*}%
		Then, $F$ is concave in $\lambda$ if ${r}$ is decreasing in $t$.
		Next, observe that 
		\[\log  F(t,\lambda ,\alpha )=\log \left(1-\bar H^{\alpha
		}\left( \lambda t\right)\right).\]
		Then 
		\begin{equation*}
		\frac{\partial \log  F(t,\lambda ,\alpha )}{\partial \lambda }=
		\alpha t%
		{r}\left( \lambda t\right)\frac{\bar H^{\alpha
			}\left( \lambda t\right)}{1-\bar H^{\alpha
			}\left( \lambda t\right)}.
		\end{equation*}%
		Thus, the above ratio is decreasing if $ r$ is decreasing since 
		$\bar H^{\alpha
		}\left( \lambda t\right)/(1-\bar H^{\alpha 
		}\left( \lambda t\right))$ is positive-valued and decreasing.
		Therefore, $F$ is log-concave in $\lambda$ if $ r$ decreasing.
		Now, we study the
		monotony of $\log \bar F(t,\lambda ,\alpha )$ with respect to $%
		\lambda $. It is easy to obtain
		$\log \bar F(t,\lambda ,\alpha )=\alpha \log \bar H(\lambda t)$, then
		\begin{equation*}
		\frac{\partial \log \bar F(t,\lambda ,\alpha )}{\partial \lambda }=- \alpha t%
		{r}\left( \lambda t\right) .
		\end{equation*}%
		Hence, ${r}\left( \lambda t\right) $ is decreasing (increasing) in 
		$\lambda ,$ which is equivalent to ${r}$ decreasing (increasing)
		in $t$, implies $\bar F(t,\lambda ,\alpha )$ is log-convex (log-concave) in\ $%
		\lambda $. 
		\item[vi)] Let us consider $F$ defined as in \eqref{pareto}.
		It is easy to
		obtain%
		\begin{equation*}
		\frac{\partial F(t,\lambda ,\alpha ,\beta )}{\partial \lambda }=\alpha \beta
		t\frac{\left( \lambda t\right) ^{\alpha -1}}{\left( 1+\left( \lambda
			t\right) ^{\alpha }\right) ^{\beta +1}}\geq 0, \quad \quad
		\frac{\partial F(t,\lambda ,\alpha ,\beta )}{\partial \alpha }=\beta \left(
		\log \lambda t\right) \frac{\left( \lambda t\right) ^{\alpha }}{\left(
			1+\left( \lambda t\right) ^{\alpha }\right) ^{\beta +1}},
		\end{equation*}%
		and
		\begin{equation*}
		\frac{\partial F(t,\lambda ,\alpha ,\beta )}{\partial \beta }=\frac{\log
			\left( 1+\left( \lambda t\right) ^{\alpha }\right) }{\left( 1+\left( \lambda
			t\right) ^{\alpha }\right) ^{\beta }}\geq 0,
		\end{equation*}%
		then $F(t,\lambda ,\alpha ,\beta )$ is increasing in $\lambda$ and in $\beta $
		and it is not monotone in $\alpha$. 
		Therefore, we can use the results of this manuscript when ${\bm\alpha }%
		\overset{m}{\geq }{\bm\alpha }^{\ast }$.
		
		Now, we
		study the monotony of the logarithm of the reliability function. 
		Observe that
		\[\log \bar F(t,\lambda ,\alpha ,\beta )=
		-\beta \log \left(1+\left( \lambda
		t\right) ^{\alpha }\right).\]
		Then,
		\begin{equation*}
		\frac{\partial \log \bar{F}(t,\lambda
			,\alpha ,\beta )}{\partial \lambda }=-t\alpha \beta \frac{\left( \lambda
			t\right) ^{\alpha -1}}{1+\left( \lambda t\right) ^{\alpha }}
		\end{equation*}%
		and
		\begin{equation*}
		\frac{\partial^{2} \log \bar{F}(t,\lambda
			,\alpha ,\beta )}{\partial \lambda^{2} }=\frac{%
			\alpha \beta }{\lambda ^{2}}\frac{\left( \lambda t\right) ^{\alpha }}{\left(
			1+\left( \lambda t\right) ^{\alpha }\right) ^{2}}\left( 1+\left( \lambda
		t\right) ^{\alpha }-\alpha \right) \geq 0,
		\end{equation*}%
		if $0<\alpha \leq 1$. In other words, if $0<\alpha \leq 1$ then $\bar{F}%
		(t,\lambda ,\alpha ,\beta )$ is log-convex in $\lambda $ and therefore, $%
		\bar{F}(t,\lambda ,\alpha ,\beta )$ is convex in $\lambda $. 
		
		On the other hand, 
		\begin{equation*}
		\frac{\partial \log \bar{F}(t,\lambda
			,\alpha ,\beta )}{\partial \alpha }=-\frac{\beta \left( \lambda t\right)
			^{\alpha }\left( \log \lambda t\right) }{1+\left( \lambda t\right) ^{\alpha }}%
		,
		\end{equation*}%
		and
		\begin{equation*}
		\frac{\partial^{2} \log \bar{F}(t,\lambda
			,\alpha ,\beta )}{\partial \alpha^{2} }=-\beta
		\left( \log \lambda t\right) ^{2}\frac{\left( \lambda t\right) ^{\alpha }}{%
			\left( 1+\left( \lambda t\right) ^{\alpha }\right) ^{2}}\leq 0,
		\end{equation*}%
		therefore $\bar{F}(t,\lambda ,\alpha ,\beta )$ is log-concave in $%
		\alpha $ for any $\lambda, \beta >0$.

		Finally, for the parameter $\beta $, 
		\begin{equation*}
		\frac{\partial \log \bar{F}(t,\lambda
			,\alpha ,\beta )}{\partial \beta }=-\log \left( 1+\left( \lambda t\right)
		^{\alpha }\right),
		\end{equation*}%
		which is constant with respect to $\beta$, and therefore $\bar F$ is both
		log-convex and log-concave in $\beta$.
		\item[vii)] 
		Firstly, we study the monotony and the
		convexity/concavity of the distribution function defined in \eqref{EKG} with respect to $\alpha $.
		Observe that%
		\begin{equation*}
		\frac{\partial F(t,\alpha ,\beta ,\gamma )}{\partial \alpha }=\gamma \beta
		\left( \log H\left( t\right) \right) H^{\alpha }\left( t\right) \left(
		1-H^{\alpha }\left( t\right) \right) ^{\beta -1}\left( 1-\left( 1-H^{\alpha
		}\left( t\right) \right) ^{\beta }\right) ^{\gamma -1}\leq 0,
		\end{equation*}%
		so $F(t,\alpha ,\beta ,\gamma )$ is decreasing in $\alpha $. Now, we have
		that the second partial derivative with respect to $\alpha $ is 
		\begin{eqnarray*}
			\frac{\partial ^{2}F(t,\alpha ,\beta ,\gamma )}{\partial \alpha ^{2}}%
			&\overset{sign}{=}&\left( \log H\left( t\right) \right) ^{2}H^{\alpha }\left(
			t\right) \left( 1-H^{\alpha }\left( t\right) \right) ^{\beta -1}\left(
			1-\left( 1-H^{\alpha }\left( t\right) \right) ^{\beta }\right) ^{\gamma -1}
			\\
			&&-\left( \beta -1\right) \left( \log H\left( t\right) H^{\alpha }\left(
			t\right) \right) ^{2}\left( 1-H^{\alpha }\left( t\right) \right) ^{\beta
				-2}\left( 1-\left( 1-H^{\alpha }\left( t\right) \right) ^{\beta }\right)
			^{\gamma -1} \\
			&&+\beta \left( \gamma -1\right) \left( \log H\left( t\right) H^{\alpha
			}\left( t\right) \right) ^{2}\left( \left( 1-H^{\alpha }\left( t\right)
			\right) ^{\beta -1}\right) ^{2}\left( 1-\left( 1-H^{\alpha }\left( t\right)
			\right) ^{\beta }\right) ^{\gamma -2} \\
			&\overset{sign}{=}&\left( 1-H^{\alpha }\left( t\right) \right) ^{\beta
				-1}\left( 1-\left( 1-H^{\alpha }\left( t\right) \right) ^{\beta }\right)
			^{\gamma -1}\\
			&&-\left( \beta -1\right) H^{\alpha }\left( t\right) \left(
			1-H^{\alpha }\left( t\right) \right) ^{\beta -2}\left( 1-\left( 1-H^{\alpha
			}\left( t\right) \right) ^{\beta }\right) ^{\gamma -1} \\
			&&+\beta \left( \gamma -1\right) H^{\alpha }\left( t\right) \left( \left(
			1-H^{\alpha }\left( t\right) \right) ^{\beta -1}\right) ^{2}\left( 1-\left(
			1-H^{\alpha }\left( t\right) \right) ^{\beta }\right) ^{\gamma -2}.
		\end{eqnarray*}%
		Set $x=1-H^{\alpha }\left( t\right)$, then%
		\begin{eqnarray*}
			\frac{\partial ^{2}F(t,\alpha ,\beta ,\gamma )}{\partial \alpha ^{2}}%
			&\overset{sign}{=}&x^{\beta -1}\left( 1-x^{\beta }\right) ^{\gamma -1}-\left(
			\beta -1\right) \left( 1-x\right) x^{\beta -2}\left( 1-x^{\beta }\right)
			^{\gamma -1}\\
			&&\hspace{0.5cm}+\beta \left( \gamma -1\right) \left( 1-x\right) \left( x^{\beta
				-1}\right) ^{2}\left( 1-x^{\beta }\right) ^{\gamma -2} \\
			&=&\left( 1-x^{\beta }\right) ^{\gamma }\left( \frac{x^{\beta -1}}{%
				1-x^{\beta }}-\left( \beta -1\right) \frac{\left( 1-x\right) x^{\beta -2}}{%
				1-x^{\beta }}+\beta \left( \gamma -1\right) \frac{\left( 1-x\right) \left(
				x^{\beta -1}\right) ^{2}}{\left( 1-x^{\beta }\right) ^{2}}\right)  \\
			&\overset{sign}{=}&\frac{x^{\beta -1}}{1-x^{\beta }}\left( 1+\left( 1-\beta
			\right) \frac{1-x}{x}\right) +\beta \left( \gamma -1\right) \frac{\left(
				1-x\right) \left( x^{\beta -1}\right) ^{2}}{\left( 1-x^{\beta }\right) ^{2}}
			\geq 0,
		\end{eqnarray*}%
		if $0<\beta \leq 1$ and $0<\gamma \leq 1$, i.e., $F(t,\alpha ,\beta ,\gamma )
		$ is convex in $\alpha$ when $0<\beta,\gamma \leq 1$.
		
		Secondly, we study the monotony and the
		convexity/concavity of the distribution function with respect to $\beta $.
		In this case,%
		\begin{equation*}
		\frac{\partial F(t,\alpha ,\beta ,\gamma )}{\partial \beta }=-\gamma \left(
		1-\left( 1-H^{\alpha }(t)\right) ^{\beta }\right) ^{\gamma -1}\left(
		1-H^{\alpha }(t)\right) ^{\beta }\log \left( 1-H^{\alpha }(t)\right) \geq 0,
		\end{equation*}%
		then $F(t,\alpha ,\beta ,\gamma )$ is increasing in $\beta $. On the other
		hand, the second partial derivative is%
		\begin{eqnarray*}
			\frac{\partial ^{2}F(t,\alpha ,\beta ,\gamma )}{\partial \beta ^{2}}%
			&\overset{sign}{=}&\left( \gamma -1\right) \left( 1-\left( 1-H^{\alpha
			}(t)\right) ^{\beta }\right) ^{\gamma -2}\left( \left( 1-H^{\alpha }(t)\right)
			^{\beta }\log \left( 1-H^{\alpha }(t)\right) \right) ^{2}- \\
			&&-\left( 1-\left( 1-H^{\alpha }(t)\right) ^{\beta }\right) ^{\gamma -1}\left(
			1-H^{\alpha }(t)\right) ^{\beta }\left( \log \left( 1-H^{\alpha }(t)\right) \right)
			^{2} \\
			&\overset{sign}{=}&\left( \gamma -1\right) \left( 1-H^{\alpha }(t)\right)
			^{\beta }-\left( 1-\left( 1-H^{\alpha }(t)\right) ^{\beta }\right)  \\
			&=&\gamma \left( 1-H^{\alpha }(t)\right) ^{\beta }-1\leq 0,
		\end{eqnarray*}%
		if $0<\gamma \leq 1$. Hence, $F(t,\alpha ,\beta ,\gamma )$ is concave when $%
		\gamma \leq 1$.

		Thirdly, we study the monotony and the convexity/concavity of
		the distribution function with respect to $\gamma $. In this case,%
		\begin{equation*}
		\frac{\partial F(t,\alpha ,\beta ,\gamma )}{\partial \gamma }=\left(
		1-\left( 1-H^{\alpha }(t)\right) ^{\beta }\right) ^{\gamma }\log \left( 1-\left(
		1-H^{\alpha }(t)\right) ^{\beta }\right) \leq 0,
		\end{equation*}%
		and the function $F(t,\alpha ,\beta ,\gamma )$ is decreasing in $\gamma $.
		Taking the second partial derivative with respect to $\gamma $, we get%
		\begin{equation*}
		\frac{\partial F(t,\alpha ,\beta ,\gamma )}{\partial \gamma }=\left(
		1-\left( 1-H^{\alpha }(t)\right) ^{\beta }\right) ^{\gamma }\left( \log \left(
		1-\left( 1-H^{\alpha }(t)\right) ^{\beta }\right) \right) ^{2}\geq 0,
		\end{equation*}%
		so $F(t,\alpha ,\beta ,\gamma )$ is convex in $\gamma $ for any $\alpha,\beta>0$. 
		\item[viii)] Let us consider the distribution function defined in \eqref{ELS}.
		It
		is easy to check that $F(t,\lambda ,\alpha ,\theta )$ is decreasing in $%
		\alpha $ and%
		\begin{equation*}
		\frac{\partial \log F(t,\lambda ,\alpha ,\theta )}{\partial \alpha }=\log
		H\left( \frac{t-\lambda }{\theta }\right) 
		\end{equation*}%
		is a constant function with respecto to $\alpha $. Next, we study the
		monotony and the log-convexity/log-concavity of the distribution function
		with respect to $\lambda $. In this case,%
		\begin{equation*}
		\frac{\partial F(t,\lambda ,\alpha ,\theta )}{\partial \lambda }=-\frac{%
			\alpha }{\theta }H^{\alpha }\left( \frac{t-\lambda }{\theta }\right) 
		\tilde{r}\left( \frac{t-\lambda }{\theta }\right) \leq 0,
		\end{equation*}%
		where $\tilde{r}$ is the reversed hazard rate of $H$. Then, $F(t,\lambda
		,\alpha ,\theta )$ is decreasing in $\lambda $. On the other hand, we have 
		\begin{equation*}
		\frac{\partial \log F(t,\lambda ,\alpha ,\theta )}{\partial \lambda }=-\frac{%
			\alpha }{\theta }\tilde{r}\left( \frac{t-\lambda }{\theta }\right) .
		\end{equation*}%
		Therefore, $F(t,\lambda ,\alpha ,\theta )$ is log-concave in $\lambda $ if $%
		\tilde{r}\left( \frac{x-\lambda }{\theta }\right) $ is increasing in $%
		\lambda $, which is equivalent to $\tilde{r}$ decreasing in $t$, since 
		\begin{equation*}
		\frac{\partial \tilde{r}\left( \frac{t-\lambda }{\theta }\right) }{%
			\partial \lambda }=-\frac{1}{\theta }\tilde{r}^{\prime }\left( \frac{%
			t-\lambda }{\theta }\right) \text{.}
		\end{equation*}%

		Finally, we investigate the monotony and the log-convexity/log-concavity of
		the distribution function with respect to $\theta $. It is easy to check that%
		\begin{equation*}
		\frac{\partial F(t,\lambda ,\alpha ,\theta )}{\partial \theta }=-\alpha
		\left( \frac{t-\lambda }{\theta ^{2}}\right) H^{\alpha }\left( \frac{%
			t-\lambda }{\theta }\right) \tilde{r}\left( \frac{t-\lambda }{\theta }%
		\right) \leq 0,
		\end{equation*}%
		i.e., $F(t,\lambda ,\alpha ,\theta )$ is decreasing in $\theta $. Now, we
		have%
		\begin{equation*}
		\frac{\partial \log F(t,\lambda ,\alpha ,\theta )}{\partial \theta }=-\alpha
		\left( \frac{t-\lambda }{\theta ^{2}}\right) \tilde{r}\left( \frac{%
			t-\lambda }{\theta }\right) \overset{sign}{=}-\left( \frac{t-\lambda }{%
			\theta }\right) ^{2}\tilde{r}\left( \frac{t-\lambda }{\theta }\right) .
		\end{equation*}%
		Hence, $\log F(t,\lambda ,\alpha ,\theta )$ is convex (concave) in $\theta $ if $%
		\left( \frac{t-\lambda }{\theta }\right) ^{2}r\left( \frac{t-\lambda }{%
			\theta }\right) $ is decreasing (increasing) in $\theta $. Note that%
		\begin{equation*}
		\frac{\partial }{\partial \theta }\left( \left( \frac{t-\lambda }{\theta }%
		\right) ^{2}\tilde{r}\left( \frac{t-\lambda }{\theta }\right) \right) =-%
		\frac{\left( t-\lambda \right) ^{2}}{\theta ^{3}}\left( 2\tilde{r}\left( 
		\frac{t-\lambda }{\theta }\right) +\frac{t-\lambda }{\theta }\tilde{r}%
		^{\prime }\left( \frac{t-\lambda }{\theta }\right) \right) \leq (\geq) 0
		\end{equation*}%
		if $x\tilde{r}\left( x\right) $ is increasing (decreasing) in $x$.
		Therefore, $F(t,\lambda ,\alpha ,\theta )$ is log-convex (log-concave) in $\theta $ when
		$x\tilde{r}\left( x\right) $ is increasing (decreasing) in $x$.
		
	\end{itemize}
\end{proof}

\end{document}